\documentclass[10pt]{article}
\usepackage[top=3.5cm, bottom=3cm, left=3cm, right=2.5cm]{geometry} 

\usepackage{amssymb,amsmath,amsfonts, amsthm}
\numberwithin{equation}{section}
\usepackage{mathtools}
\usepackage{thmtools} 

\usepackage{authblk}
\usepackage{hyperref}
\usepackage{cleveref}

\usepackage[shortlabels]{enumitem} 
\usepackage{graphicx}
\usepackage{subcaption}
\graphicspath{ {./figures/} }

\usepackage{mathrsfs}
\usepackage{bm}

\usepackage[backend=bibtex]{biblatex}
\theoremstyle{plain}
\newtheorem{theorem}{Theorem}[section]
\theoremstyle{plain}
\newtheorem{lemma}[theorem]{Lemma}
\newtheorem{proposition}[theorem]{Proposition}
\newtheorem{corollary}[theorem]{Corollary}
\theoremstyle{definition}
\newtheorem{definition}[theorem]{Definition}
\newtheorem{remark}[theorem]{Remark}
\newtheorem{assumption}[theorem]{Assumption}

\newcounter{ass}

\newenvironment{ass}
{%
	\refstepcounter{ass}%
	\crefalias{equation}{ass}%
	\begin{equation}
	}
	{%
	\end{equation}%
}

\newcommand{\eps}{\ensuremath{\varepsilon}}

\newcommand{\cD}{\mathcal{D}}

\newcommand{\cF}{\mathcal{F}}
\newcommand{\cG}{\mathcal{G}}

\newcommand{\cO}{\mathcal{O}}

\newcommand{\cS}{\mathcal{S}}

\newcommand{\bN}{\mathbb{N}}

\newcommand{\bP}{\mathbb{P}}

\newcommand{\bR}{\mathbb{R}}

\newcommand{\bZ}{\mathbb{Z}}

\newcommand{\mfv}{\mathfrak{v}}

\newcommand{\mfu}{\mathfrak{u}}
\newcommand{\bgg}{\bm{\gamma}}

\newcommand\dif{\mathop{}\!\mathrm{d}}

\DeclarePairedDelimiterX{\inp}[2]{\langle}{\rangle}{#1, #2}
\DeclarePairedDelimiter\abs{\lvert}{\rvert}
\DeclarePairedDelimiter\norm{\lVert}{\rVert}

\DeclareMathOperator*{\supp}{supp}

\crefname{ass}{assumption}{assumptions}
\Crefname{ass}{Assumption}{Assumptions}

\usepackage{xcolor}

\title{Higher-Order QMC Galerkin FEM for Singular Perturbation:
	\\ 
	Reaction--Diffusion with Affine-Parametric Reaction}
\date{June 29, 2026}

\author[$\dagger$]{F. Rohner}
\author[$\dagger$]{Ch. Schwab} 

\affil[$\dagger$]{
	\footnotesize Seminar for Applied Mathematics, ETH Z\"{u}rich\\ 
	R\"{a}mistrasse 101, 8092 Z\"urich, Switzerland\\
	\texttt{\{fabian.rohner,christoph.schwab\}@sam.math.ethz.ch}
}

\begin{document}
\maketitle

{\small
\begin{abstract}
We analyze higher-order quasi--Monte Carlo (QMC) Galerkin discretizations for a
singularly perturbed reaction--diffusion problem with an affine-parametric
reaction coefficient. We derive error estimates that are robust with respect to the
perturbation parameter \(\varepsilon\). 
To this end, we establish \(\varepsilon\)-uniform 
stability and mixed-parametric regularity estimates in an
\(\varepsilon\)-weighted energy norm, with constants independent of the parameter dimension and of \(\varepsilon\). 
These estimates imply robust dimension truncation estimates and
dimension-independent higher-order QMC error bounds for linear quantities of interest. 
Assuming an abstract, \(\varepsilon\)-robust finite element approximation property, 
we combine the parametric estimates with the spatial discretization analysis 
to derive a general fully discrete QMC--FE error estimate with constants independent 
of both the parameter dimension and \(\varepsilon\). 
We verify this approximation property for a one-dimensional model problem 
using continuous piecewise affine finite elements on Shishkin meshes. 
Numerical experiments illustrate the predicted robustness and convergence rates.
\end{abstract}
}

\section{Introduction}
\label{sec:Intro}
\emph{Singularly perturbed PDEs} 
arise in numerous areas of science and engineering, 
typically as a consequence of (nondimensional) parameters in the PDE data 
that are small or large relative to the remaining inputs.
Examples occur in solid mechanics, in the study of so-called ``thin solids'' such as beams, plates, and shells.
Here, the small parameter $\eps$ is the nondimensional ratio of the 
thickness to the diameter of the structure under consideration.
Further examples include viscous incompressible flow, where the relevant large parameter is the Reynolds number, and eddy current problems in time-harmonic electromagnetics, 
where the perturbation parameter occurs in the dielectric constants of thin 
coatings of conductors.

Solutions of singularly perturbed elliptic boundary value problems are known to exhibit local elliptic regularity that is generally not uniform with respect to the perturbation parameter. Such solution families typically admit an additive decomposition into a
so-called ``smooth part'' which satisfies standard elliptic regularity estimates 
\emph{uniformly with respect to the perturbation parameter}, and
so-called ``layer parts'' which typically \emph{decay exponentially} 
with (a power of) the 
distance to the boundary of the domain and/or to the supports of, for example, source terms.

\emph{Uncertainty Quantification (UQ for short)} in PDE models of science and 
engineering has attracted considerable attention over the past decade. 

From a mathematical perspective, this requires the numerical solution of 
PDEs for so-called ``ensembles'' of input data, 
together with the characterization of the resulting outputs, 
in particular quantities of interest (QoIs), in terms of the uncertain inputs.

The aim of this work is to analyze elliptic singularly perturbed reaction--diffusion problems with an uncertain reaction coefficient. Let $0<\eps\leq 1$ denote the nondimensional singular perturbation parameter. We seek to establish $\eps$-uniform convergence rates, free from the curse of dimensionality, for quasi--Monte Carlo (QMC) integration combined with finite element (FE) approximations of a model linear elliptic parametric singularly perturbed reaction--diffusion problem.

As is well known,
singular perturbations in elliptic reaction--diffusion PDEs 
induce so-called \emph{boundary layers}, 
i.e. 
exponentially decaying solution components in the physical domain $D$.
This prevents direct application of elliptic regularity theory in Sobolev spaces,
and obstructs standard convergence bounds for finite element methods. 
While finite element methods for 
singularly perturbed reaction--diffusion and convection--diffusion 
problems are well-studied (see, e.g., \cite{RST2nd,Linss1985} and the references there), 
mathematical error analysis of such PDEs with uncertain inputs 
remains largely unexplored. 
In particular, the stability constants, norm equivalences, and adjoint regularity estimates 
required to analyze QMC integration all depend on the perturbation parameter $\eps$, 
and uniform bounds as $\eps \to 0$ are not available in the existing QMC literature. 
Establishing $\eps$-robust regularity with respect to the spatial variable 
and 
with respect to the parametric variables is therefore a central obstacle which is 
addressed in the present work.

The remainder of this introduction presents the parametric singularly perturbed reaction--diffusion model problem under consideration and lists the structural assumptions required for the analysis.

\paragraph{Problem Setting}
\label{sec:set}

Let $D \subset \bR^d$ be a bounded domain 
and let $U\coloneqq[-1/2,1/2]^{\bN}$.
For a nondimensional parameter 
$0 < \eps \leq 1$, and for $\bm{y}\in U$, we consider the  
\emph{parametric, linear, singularly perturbed, elliptic reaction--diffusion boundary value problem} (BVP): 
find $u_{\eps}(\bm{y})$ such that
\begin{align}
	\mathcal{L}_{\eps} u_{\eps}(\bm{y}) \coloneqq
	-\eps^2 \Delta_x u_{\eps}(x; \bm{y})
	+ b(x; \bm{y}) u_{\eps}(x; \bm{y})
	&= f(x), \quad x \in D,
	\label{eq:de}
	\\
	u_{\eps}|_{\partial D}
	&= 0,
	\label{eq:bc}
\end{align}
where 
$b(x; \bm{y})$ is a parametric reaction coefficient depending on $x\in D$ 
and on the parameter $\bm{y} \in U$, 
and $f \colon D \to \bR$ is a prescribed fixed source term.

Specifically, we consider the model in which the parametric reaction coefficient $b(x;\bm{y})$ 
depends affinely on the parameter $\bm{y} \in U$, i.e.
\begin{ass}\label{ass:affine-param}
	b(x; \bm{y}) \coloneqq \overline{b}(x) + \sum_{j \geq 1} y_j \psi_j(x), \qquad x \in D, \quad \bm{y} \in U,
\end{ass}
for some functions $\overline{b} \in L^{\infty}(D)$ and $\psi_j \in L^\infty(D)$, $j \geq 1$.

To ensure that the series is well-defined for a.e. $x \in D$ and $\bm{y} \in U$, we assume
\begin{ass}\label{ass:summability}
	\sum_{j \geq 1} \norm{\psi_{j}}_{L^{\infty}(D)} < \infty.
\end{ass}
We shall later impose additional smoothness conditions on $\overline{b}$ and $\{\psi_j\}_{j \ge 1}$, as needed.

Moreover, we assume uniform bounds on the reaction coefficients 
\begin{ass}\label{ass:bbounds}
0 < b_{0} \leq b(x; \bm{y}) \leq	b_{1} < \infty, \qquad \mathrm{a.e.}\; x \in D, \quad \bm{y} \in U,
\end{ass}
for some constants $b_{0}, b_{1}$.

	
The quantity of interest is the expectation of a linear functional of the
solution. Assuming that the parameter $\bm{y}$ is sampled from the uniform distribution on $U$,
this amounts to computing
\begin{equation}
	\label{eq:integral-functional-sol}
	I(G(u_{\eps}))
	\coloneqq \int_U G(u_{\eps}(\cdot;\bm y))\,\bm{\pi}(\dif\bm y),
\end{equation}
where $G$ is an $\eps$-independent, bounded linear functional on the relevant solution space 
and $\bm{\pi}$ denotes the uniform probability measure on $U$.
The infinite-dimensional parameter dependence is handled by truncating the expansion
in \cref{ass:affine-param} after $s$ terms. If $u_{\eps,s}^h$ denotes
the finite element approximation of the corresponding truncated solution, then
we approximate the $s$-dimensional integral~\eqref{eq:integral-functional-sol} by a QMC quadrature rule
\begin{equation}
	\label{eq:qmc-rule-sol}
	Q_{N,s}(G(u_{\eps,s}^h))
	\coloneqq \frac1N\sum_{n=0}^{N-1}
	G\left(u_{\eps,s}^h\left(\cdot;\bm y_n-\frac{\bm{1}}{\bm{2}}\right)\right),
\end{equation}
where the points $\bm{y}_{0},\dots, \bm{y}_{N-1} \in [0,1]^s$ are chosen according to a higher-order lattice rule.
	
A central ingredient (e.g. \cite{kuo2012QMC,dick2014higher,Kuo2015Multilevel})
for the analysis of the QMC error is the summability of the fluctuation coefficients $\psi_j$ in \eqref{ass:affine-param}.
For this reason, we will later strengthen~\eqref{ass:summability} 
by assuming the existence of $0 < p < 1$, 
such that
\begin{ass}\label{ass:p-summability}
	\sum_{j \geq 1} \norm{\psi_{j}}^p_{L^{\infty}(D)} < \infty.
\end{ass}
We also impose the \emph{ordering}
\begin{ass}\label{ass:ordering}
	\norm{\psi_1}_{L^\infty(D)} \geq \norm{\psi_2}_{L^\infty(D)} \geq \cdots \ge 0.
\end{ass}

To approximate the parametric solutions $u_{\eps}(\bm{y})$, $\bm{y} \in U$, 
in the physical domain $D$ using a finite element discretization, 
one typically requires additional smoothness of the coefficient $b(\cdot; \bm{y})$ 
with respect to the spatial variable $x$ in order to obtain 
$\eps$-robust convergence of the FE approximation.

A sufficient criterion for $b(\cdot; \bm{y}) \in H^k$, $k \geq 1$, 
is to assume in addition to \cref{ass:summability,ass:p-summability} 
that $\overline{b} \in H^k(D)$ and $(\norm{\psi_{j}}_{H^k(D)})_{j \ge 1} \in \ell^1(\bN)$.
%
See also \cite{cohen2010Nterm} for a more detailed discussion of possible choices for bases and expansions of the coefficient $b$.

\paragraph{Contributions}	
The analysis in this paper builds on the established QMC theory for elliptic PDEs with parametric coefficients. In particular, parametric regularity estimates were developed in \cite{cohen2010Nterm}, their use in QMC error analysis was studied in \cite{kuo2012QMC}, and higher-order QMC extensions were obtained in \cite{dick2014higher}. Dimension truncation estimates were sharpened in \cite{gantner2018dimension}, with further extensions to log-normal coefficients and nonlinear parameter dependence in \cite{guth2024lognormal}; see also the survey \cite{kuo2016survey}.

The main contribution of this work is to show that this framework can be made robust with respect to the singular perturbation parameter $\eps$ for the reaction--diffusion problem \eqref{eq:de}--\eqref{eq:bc}. More precisely, we establish $\eps$-uniform stability and mixed-parametric regularity estimates in a weighted energy norm, which in turn yield dimension truncation bounds with constants independent of $\eps\in(0,1]$. Combining these regularity estimates with higher-order QMC rules based on SPOD weights, we establish a dimension-independent QMC convergence rate with constants uniform in $\eps$. Moreover, the QMC construction is itself independent of $\eps$: a single generating vector, and hence a single QMC point set, can be used for all $\eps\in(0,1]$. Finally, we formulate an abstract robust spatial approximation assumption and verify it for a one-dimensional model problem using continuous piecewise affine finite elements on Shishkin meshes. This yields a fully discrete QMC--FE error bound that is uniform in $\eps$, the truncation dimension $s$, the number of QMC points $N$, and the spatial discretization parameter.
	
To the best of our knowledge, 
this is the first QMC--FE error analysis for randomly parametrized singularly 
perturbed reaction--diffusion problems with constants fully robust in $\eps$. 
Related QMC--FE analyses for singularly perturbed convection--diffusion problems with random velocity fields are available, but concern a different model class and employ $\eps$-dependent QMC weights. In contrast, for the reaction--diffusion problem considered here, the QMC weights are independent of $\eps$, so that a single family of QMC point sets can be used uniformly for all $\eps\in(0,1]$.
\paragraph{Structure}
This text is structured as follows: 
\cref{sec:param-reg} establishes $\eps$-uniform parametric regularity results for the model problem. 
\Cref{sec:dim-trunc} provides a brief review of a
dimension truncation bound based on the $p$-summability \cref{ass:p-summability}, 
which we verify to carry over directly to the $\eps$-weighted norm $\norm{\cdot}_{1,\eps,D}$ 
with constants that are uniform with respect to the perturbation parameter $\eps$. \Cref{sec:qmc} 
recalls the construction of higher-order QMC rules and 
corresponding error bounds based on the parametric regularity estimates. 
\Cref{sec:spatial-approx} states a spatial discretization
property and verifies it for a Shishkin finite
element method in one space dimension. \Cref{sec:combined-error} combines
the three error sources into a fully discrete estimate, and \cref{sec:numerics} presents numerical experiments.

\paragraph{Notation}	
We finish the introduction with notation used throughout the paper. 
Let $\cF$ be
the set of all finitely supported multi-indices
$\bm\nu=(\nu_j)_{j\geq1}\in\bN_0^{\bN}$. For $\bm\nu\in\cF$, set
\[
	\supp(\bm\nu)\coloneqq\{j\in\bN:\nu_j\neq0\},
	\qquad
	\abs{\bm\nu}\coloneqq\sum_{j\geq1}\nu_j,
	\qquad
	\bm\nu!\coloneqq\prod_{j\geq1}\nu_j!.
\]
For two multi-indices $\bm\beta,\bm\nu\in\cF$, we write
$\bm\beta\leq\bm\nu$ if $\beta_j\leq\nu_j$ for all $j$, and define
\[
	\binom{\bm\nu}{\bm\beta}
	\coloneqq \prod_{j\geq1}\binom{\nu_j}{\beta_j},
	\qquad
	\bm\gamma^{\bm\nu}
	\coloneqq \prod_{j\geq1}\gamma_j^{\nu_j}.
\]

Our analysis is carried out using an $\eps$-weighted norm: For $\eps \in (0,1]$, we define the norm $\norm{\cdot}_{1,\eps,D}$ on $H_0^1(D)$ given by
\begin{equation}\label{eq:def-eps-norm}
	\norm{v}_{1,\eps,D}^2 \coloneqq \eps^2 \norm{\nabla v}_{L^2(D)}^2 + \norm{v}_{L^2(D)}^2,
\end{equation}
and let 
\begin{equation}\label{eq:def-V-eps}
	V_\eps \coloneqq \overline{C_0^\infty(D)}^{\norm{\cdot}_{1,\eps,D}}.
\end{equation}
For $\eps>0$, this space is $H_0^1(D)$ as a set, but with an
$\eps$-dependent norm equivalence.

Using an argument similar to \cite[Prop. 8.14]{brezis2011functional}, 
in the case $D \subset \bR$ the norm on the dual space $V_{\eps}'$ 
can be characterized as follows:
with $(\cdot, \cdot)$ denoting the $L^2(D)$ inner product,
\begin{equation*}
	\forall f \in V_{\eps}'\; \exists g_0, g_1 \in L^2(D): \; \inp{f}{v}_{V_{\eps}' \times V_{\eps}} = (g_0(f), v)  + (g_1(f), v'), \qquad \forall v \in V_{\eps},
\end{equation*}
and
\begin{equation*}
	\norm{f}_{V_{\eps}'} \coloneqq \sup_{v \in V_{\eps}} \frac{\abs{\inp{f}{v}_{V_{\eps}' \times V_{\eps}}}}{\norm{v}_{1,\eps,D}} \simeq \norm{g_0(f)}_{L^2(D)} + \eps^{-1} \norm{g_1(f)}_{L^2(D)},
\end{equation*}
where the constant hidden in $\simeq$ is independent of $\eps$.
In the case $\eps = 0$, $V_{\eps}$ and $V_{\eps}'$ each reduce to $L^2(D)$. 
This motivates the following notation, for all $\eps \in [0,1]$:
\begin{equation*}
	\norm{f}_{-1,\eps,D} \coloneqq \norm{f}_{V_{\eps}'}, \qquad \forall f \in V_{\eps}'.
\end{equation*}

Throughout, $C>0$ denotes a generic, positive constant whose values can change 
at each use. 
Its value is, however, always independent of $\eps$ and of $s$, the parameter
dimension in numerical simulation.

\section{Robust parametric regularity}
\label{sec:param-reg}
In this section we establish the basic well-posedness, stability, and parametric
regularity estimates for the singularly perturbed reaction--diffusion problem
\eqref{eq:de}--\eqref{eq:bc}. The central point is that all estimates are
uniform with respect to the perturbation parameter $\eps \in (0,1]$ when the
solution is measured in the weighted norm $\norm{\cdot}_{1,\eps,D}$.
\subsection{Parametric weak formulation}
\label{subsec:param-weak-formulation}
For fixed $\eps \in (0,1]$ and $\bm{y} \in U$, the weak formulation of
\eqref{eq:de}--\eqref{eq:bc} is: find $u_{\eps}(\cdot;\bm{y}) \in V_{\eps}$
such that
\begin{equation}
	\label{eq:param-weak-problem}
	\eps^2 \int_D \nabla u_{\eps}(x;\bm{y}) \cdot \nabla v(x) \dif x
	+ \int_D b(x;\bm{y}) u_{\eps}(x;\bm{y}) v(x) \dif x
	= \inp{f}{v}_{V_{\eps}'\times V_{\eps}},
	\qquad \forall v \in V_{\eps}.
\end{equation}
When $f\in L^2(D)$, the right-hand side is simply
$(f,v)_{L^2(D)}$.

We write the associated parametric bilinear forms as
\begin{equation*}
	B_{\eps}(u,v; \bm{y}) \coloneqq \eps^2 \int_{D} \nabla u(x) \cdot \nabla v(x) \dif x + \int_{D} b(x; \bm{y}) u(x) v(x) \dif x, \qquad u,v \in V_{\eps}.
\end{equation*}
By \cref{ass:bbounds}, this bilinear form is continuous and coercive,
uniformly in $\eps$ and $\bm{y}$. Indeed, for all $u,v\in V_{\eps}$,
\begin{align*}
	\abs{B_{\eps}(u,v;\bm{y})}
	&\leq \eps^2 \norm{\nabla u}_{L^2(D)}\norm{\nabla v}_{L^2(D)}
	+ b_1 \norm{u}_{L^2(D)}\norm{v}_{L^2(D)} \\
	&\leq (1+b_1)\norm{u}_{1,\eps,D}\norm{v}_{1,\eps,D},
\end{align*}
and
\begin{align*}
	B_\eps(u,u; \bm{y}) &= \eps^2 \int_{D} \abs{\nabla u(x)}^2 \dif x + \int_{D} b(x; \bm{y}) \abs{u(x)}^2 \dif x\\
	&\geq \min\left\{1, b_{0}\right\} \norm{u}_{1,\eps,D}^2.
\end{align*}
\begin{remark}
	For fixed $\bm{y}\in U$, the energy norm
	\begin{equation}
		\label{eq:def-energy-norm}
		\norm{w}_{E,D}^2
		\coloneqq B_{\eps}(w,w;\bm{y})
		= \eps^2 \int_D \abs{\nabla w(x)}^2\dif x
		+ \int_D b(x;\bm{y})w(x)^2\dif x
	\end{equation}
	is equivalent to $\norm{\cdot}_{1,\eps,D}$, with constants depending only on
	$b_0$ and $b_1$.
\end{remark}

The Lax-Milgram theorem yields the existence of a unique solution for any $f \in V_{\eps}'$ (hence also for $f \in L^2(D)$). In particular, there exists a unique element $u_{\eps}(\cdot;  \bm{y}) \in V_{\eps}$ such that
\begin{equation}\label{eq:var-form}
	B_{\eps}(u_{\eps}(\cdot;\bm{y}),v; \bm{y}) = \inp{f}{v}_{V_{\eps}' \times V_{\eps}}, \qquad \forall v \in V_{\eps}.
\end{equation} 
Moreover,
\begin{equation}\label{eq:bound-sol}
	\norm{u_{\eps}(\cdot; \bm{y})}_{1,\eps,D} \leq \frac{1}{\min\{1,b_{0}\}} \norm{f}_{-1,\eps,D} \eqqcolon B \norm{f}_{-1,\eps,D}, \qquad \forall \bm{y} \in U.
\end{equation}
\begin{remark}
	Although the norm $\norm{f}_{-1,\eps,D}$ may depend on $\eps$, the embedding
	$L^2(D)\hookrightarrow V_{\eps}'$ has norm at most one. Hence, if
	$f\in L^2(D)$, then \eqref{eq:bound-sol} implies the uniform bound
	\[
		\norm{u_{\eps}(\cdot;\bm{y})}_{1,\eps,D}
		\leq B\norm{f}_{L^2(D)}.
	\]
\end{remark}

We shall also use the following stability estimate with respect to perturbations
of the reaction coefficient.
\begin{lemma}[Stability]
	\label{lem:coefficient-stability}
	Let $b,\tilde b\in L^{\infty}(D)$ satisfy
	\[
	0<b_0\leq b,\tilde b\leq b_1<\infty \quad \text{a.e. in }D.
	\]
	For $\eps\in(0,1]$, let $u_{\eps},\tilde u_{\eps}\in V_{\eps}$ solve
	\begin{equation}
		\label{eq:var-b}
		\eps^2\int_D \nabla u_{\eps}(x) \cdot \nabla v(x) \dif x
		+ \int_D b(x) u_{\eps}(x) v(x) \dif x
		= \inp{f}{v}_{V_{\eps}'\times V_{\eps}},
		\qquad \forall v\in V_{\eps},
	\end{equation}
	and
	\begin{equation}
		\label{eq:var-btilde}
		\eps^2\int_D \nabla \tilde u_{\eps}(x) \cdot \nabla v(x) \dif x
		+ \int_D \tilde b(x) \tilde u_{\eps}(x) v(x) \dif x
		= \inp{f}{v}_{V_{\eps}'\times V_{\eps}},
		\qquad \forall v\in V_{\eps}.
	\end{equation}
	Then
	\begin{equation}
		\label{eq:stab}
		\norm{u_{\eps}-\tilde u_{\eps}}_{1,\eps,D}
		\leq
		\frac{\norm{b-\tilde b}_{L^{\infty}(D)}}{\min\{1,b_0\}^2}
		\norm{f}_{-1,\eps,D}.
	\end{equation}
\end{lemma}
\begin{proof}
	Set $w\coloneqq u_{\eps}-\tilde u_{\eps}$. Subtracting
	\eqref{eq:var-btilde} from \eqref{eq:var-b} gives
	\begin{equation*}
		\eps^2\int_D \nabla w(x) \cdot \nabla v(x) \dif x
		+ \int_D b(x) w(x) v(x) \dif x
		= -\int_D (b-\tilde b)(x) \tilde u_{\eps}(x) v(x) \dif x,
		\qquad \forall v\in V_{\eps}.
	\end{equation*}
	Thus $w$ solves the variational problem with coefficient $b$ and right-hand
	side $L(v)\coloneqq -\int_D (b-\tilde b)\tilde u_{\eps}v\dif x$. The solution
	estimate \eqref{eq:bound-sol} gives
	\[
	\norm{w}_{1,\eps,D}\leq \frac{1}{\min\{1,b_0\}}\norm{L}_{-1,\eps,D}.
	\]
	Furthermore,
	\begin{align*}
		\norm{L}_{-1,\eps,D}
		&\leq \norm{b-\tilde b}_{L^{\infty}(D)}
		\norm{\tilde u_{\eps}}_{L^2(D)}
		\leq \norm{b-\tilde b}_{L^{\infty}(D)}
		\norm{\tilde u_{\eps}}_{1,\eps,D} \\
		&\leq
		\frac{\norm{b-\tilde b}_{L^{\infty}(D)}}{\min\{1,b_0\}}
		\norm{f}_{-1,\eps,D},
	\end{align*}
	which proves \eqref{eq:stab}.
\end{proof}
\subsection{Parametric regularity for affine reaction coefficients}
\label{subsection:param-reg-affine}
We now use the affine representation from \cref{ass:affine-param} to derive bounds on mixed derivatives of the parameter-to-solution map. At boundary points of $U$, derivatives are interpreted through the affine extension in the relevant coordinate directions, which preserves uniform positivity for sufficiently small perturbations.

Following the inductive argument of \cite{cohen2010Nterm}, we establish $\eps$-robust bounds for the partial derivatives of the parameter-to-solution map
$\bm y \mapsto u_\eps(\cdot;\bm y)$. Throughout, $\eps\in(0,1]$ is arbitrary but fixed.
\begin{theorem}\label{thm:param-reg}
	Let $\eps \in (0,1]$ and let the assumptions \eqref{ass:summability} and \eqref{ass:bbounds} be satisfied. At any $\bm{y} \in U$, the mapping $\bm{y} \mapsto u_{\eps}(\cdot; \bm{y})$ admits partial derivatives $\partial_{\bm{y}}^{\bm{\nu}} u_{\eps}(\cdot; \bm{y}) \in V_{\eps}$ for every $\bm{\nu} \in \cF$.
\end{theorem}
\begin{proof}
	It suffices to show the existence of the first-order derivatives; the existence of higher-order derivatives $\partial_{\bm{y}}^{\bm{\nu}} u_{\eps}(\cdot; \bm{y})$ for $\bm{\nu} \in \cF$ can be derived inductively following the same reasoning.
	
	Fix $j \ge 1$ and $\bm{y} \in U$. We first show that the first-order partial derivatives $\partial_{y_j} u_{\eps}(\cdot; \bm{y})$ exist in $V_{\eps}$.
	For $h \neq 0$ sufficiently small, set
	\begin{equation}
		w_h(x;\bm{y}) = \frac{u_{\eps}(x; \bm{y} + h \bm{e}_j) - u_{\eps}(x; \bm{y})}{h},
	\end{equation}
	where $\bm{e}_j$ is the Kronecker sequence $(\bm{e}_j)_k = \delta_{j,k}$.
	Here $h$ is chosen such that the coefficient $b(\cdot; \bm{y} + h \bm{e}_j)$ remains bounded below by a positive constant; for
	instance this holds if
	$\abs h\norm{\psi_j}_{L^{\infty}(D)}\leq b_0/2$.
	
	Subtracting the two weak formulations and dividing by $h$,
	\begin{align*}
		0 &= \eps^2 \int_{D} \nabla u_{\eps}(x; \bm{y}+ h \bm{e}_j) \cdot \nabla v(x) \dif x - \eps^2 \int_{D} \nabla u_{\eps}(x; \bm{y}) \cdot \nabla v(x) \dif x\\
		&\qquad + \int_{D} b(x; \bm{y}+ h \bm{e}_j) u_{\eps}(x; \bm{y}+ h \bm{e}_j) v(x) \dif x - \int_{D} b(x; \bm{y}) u_{\eps}(x; \bm{y}) v(x) \dif x\\
		&= h \left[\eps^2 \int_{D} \nabla w_h(x; \bm{y}) \cdot \nabla v(x) \dif x + \int_{D} b(x; \bm{y}) w_h(x; \bm{y}) v(x) \dif x\right]\\
		&\qquad + h \int_{D} \psi_j(x) u_{\eps}(x; \bm{y}+ h \bm{e}_j) v(x) \dif x.
	\end{align*}
	This shows that $w_h \in V_{\eps}$ satisfies
	\begin{equation*}
		B_{\eps}(w_h,v;\bm y)
		= -\int_D \psi_j(x)u_{\eps}(x;\bm y+h\bm e_j)v(x)\dif x,
		\qquad \forall v\in V_{\eps}.
	\end{equation*}
	Denote the right-hand side by $L_h(v)$. The stability estimate
	\eqref{eq:stab}, applied to the coefficients
	$b(\cdot;\bm y+h\bm e_j)$ and $b(\cdot;\bm y)$, implies
	\[
	\norm{u_{\eps}(\cdot;\bm y+h\bm e_j)-u_{\eps}(\cdot;\bm y)}_{1,\eps,D}
	\leq C_j\abs h\norm{f}_{-1,\eps,D},
	\]
	where $C_j$ is independent of $h$. Hence $L_h\to L_0$ in $V_{\eps}'$, with
	\[
		L_0(v)=-\int_D \psi_j(x)u_{\eps}(x;\bm y)v(x)\dif x.
	\]
	
	By the uniform well-posedness of the variational problem with coefficient
	$b(\cdot;\bm y)$, it follows that $w_h$ converges in $V_{\eps}$ to the unique
	solution $w_0\in V_{\eps}$ of
	\begin{equation}
		\label{eq:u-eps-deriv-order-1}
		B_{\eps}(w_0,v;\bm y)
		= -\int_D \psi_j(x)u_{\eps}(x;\bm y)v(x)\dif x,
		\qquad \forall v\in V_{\eps}.
	\end{equation}
	Thus $\partial_{y_j}u_{\eps}(\cdot;\bm y)=w_0$. Iterating this argument gives
	the existence of all mixed derivatives indexed by $\bm\nu\in\cF$.
\end{proof}
\begin{theorem}\label{thm:param-reg-bounds}
	For all $\eps \in [0,1]$ and for any $\bm{\nu} \in \cF$, it holds that
	\begin{equation}\label{eq:bound-partial-deriv}
		\sup_{\bm{y} \in U} \norm{\partial_{\bm{y}}^{\bm{\nu}} u_{\eps}(\cdot; \bm{y})}_{1,\eps,D} \leq B \, \abs{\bm{\nu}}! \, \bm{b}^{\bm{\nu}} \norm{f}_{-1,\eps, D},
	\end{equation}
	with the $\eps$-independent constants 
	\begin{equation}\label{eq:b-sequence-partial-deriv}
		B \coloneqq \frac{1}{\min\{1,b_{0}\}}, \qquad b_j \coloneqq \frac{\norm{\psi_j}_{L^{\infty}(D)}}{\min\{1, b_{0}\}}, \quad j \geq 1.
	\end{equation}
\end{theorem}
\begin{remark}
	For $\eps=0$, the limiting equation is understood as
	$b(x;\bm y)u_0(x;\bm y)=f(x)$, the space $V_0$ is identified with $L^2(D)$,
	and $\norm{\cdot}_{-1,0,D}$ is the $L^2(D)$-norm.
\end{remark}
\begin{proof}
	We first consider the case $\eps \in (0,1]$. We claim that the partial derivatives $\partial_{\bm{y}}^{\bm{\nu}} u_{\eps}(x; \bm{y})$ 
	satisfy variational problems of the following recursive form
	\begin{align}
		\label{eq:partial-deriv-weak-form}
		\begin{split}
			B_{\eps}(\partial_{\bm y}^{\bm\nu}u_{\eps},v;\bm y)
			= -\sum_{j\in\supp(\bm\nu)}\nu_j
			\int_D \psi_j(x)
			\partial_{\bm y}^{\bm\nu-\bm e_j}u_{\eps}(x;\bm y)v(x)\dif x,
			\qquad \forall v\in V_{\eps}.
		\end{split}
	\end{align}
	We show the claim by induction on $\abs{\bm{\nu}}$. For $\abs{\bm{\nu}} = 1$, this is \eqref{eq:u-eps-deriv-order-1}. Let now $\bm{\nu} \in \cF$ with $\abs{\bm{\nu}} > 1$ and let $k \in \supp(\bm{\nu})$. Then, $\tilde{\bm{\nu}} \coloneqq \bm{\nu} - \bm{e}_k \in \cF$ and $\abs{\tilde{\bm{\nu}}} = \abs{\bm{\nu}} - 1 \geq 0$. By the induction hypothesis, we have, for all $v \in V_{\eps}$
	\begin{align*}
		\eps^2 \int_{D} \nabla \partial_{\bm{y}}^{\tilde{\bm{\nu}}} u_{\eps} (x;\bm{y}) \cdot \nabla v(x) \dif x + \int_{D} b(x; \bm{y}) \partial_{\bm{y}}^{\tilde{\bm{\nu}}} u_{\eps}(x; \bm{y}) v(x) \dif x\\
		+ \sum_{j \in \supp(\tilde{\bm{\nu}})} \tilde{\nu}_j \int_{D} \psi_j(x) \partial_{\bm{y}}^{\tilde{\bm{\nu}} - \bm{e}_j} u_{\eps}(x; \bm{y}) v(x) \dif x = 0,
	\end{align*}
	where $\tilde{\nu}_j = \nu_j$ if $j \neq k$ and $\tilde{\nu}_k = \nu_k - 1$.
	Differentiating with respect to $y_k$ yields 
	\begin{align*}
		0 &= \eps^2 \int_{D} \nabla \partial_{\bm{y}}^{\bm{\nu}} u_{\eps} (x;\bm{y}) \cdot \nabla v(x) \dif x + \int_{D} b(x; \bm{y}) \partial_{\bm{y}}^{\bm{\nu}} u_{\eps} (x;\bm{y}) v(x) \dif x\\
		&\qquad + \int_{D} \psi_k(x) \partial_{\bm{y}}^{\tilde{\bm{\nu}}} u_{\eps}(x; \bm{y}) v(x) \dif x  + \sum_{\substack{j \in \supp(\bm{\nu})\\ j \neq k}} \nu_j \int_{D} \psi_j(x) \partial_{\bm{y}}^{\bm{\nu} - \bm{e}_j} u_{\eps}(x; \bm{y}) v(x) \dif x\\
		&\qquad + (\nu_k - 1) \int_{D} \psi_k(x) \partial_{\bm{y}}^{\bm{\nu} - \bm{e}_k} u_{\eps}(x; \bm{y}) v(x) \dif x.
	\end{align*}
	Now, choosing $v(x) = \partial_{\bm{y}}^{\bm{\nu}} u_{\eps}(x; \bm{y}) \in V_{\eps}$, we find
	\begin{align*}
		\min\{1, b_{0}\} \norm{\partial_{\bm{y}}^{\bm{\nu}} u_{\eps}(\cdot; \bm{y})}_{1,\eps,D}^2 &\leq \eps^2 \int_{D} \abs{\nabla \partial_{\bm{y}}^{\bm{\nu}} u_{\eps}(x; \bm{y})}^2 \dif x + \int_{D} b(x; \bm{y}) \abs{\partial_{\bm{y}}^{\bm{\nu}} u_{\eps}(x; \bm{y})}^2 \dif x \\
		&= - \sum_{j \in \supp(\bm{\nu})} \nu_j \int_{D} \psi_j(x) \partial_{\bm{y}}^{\bm{\nu} - \bm{e}_j} u_{\eps}(x; \bm{y}) \partial_{\bm{y}}^{\bm{\nu}} u_{\eps}(x; \bm{y}) \dif x\\
		&\leq \sum_{j \in \supp(\bm{\nu})} \nu_j \norm{\psi_j}_{L^{\infty}(D)}  \norm{\partial_{\bm{y}}^{\bm{\nu} - \bm{e}_j} u_{\eps}(\cdot; \bm{y})}_{L^2(D)} \norm{\partial_{\bm{y}}^{\bm{\nu}} u_{\eps}(\cdot; \bm{y})}_{L^2(D)}\\
		&\leq \sum_{j \in \supp(\bm{\nu})} \nu_j \norm{\psi_j}_{L^{\infty}(D)}  \norm{\partial_{\bm{y}}^{\bm{\nu} - \bm{e}_j} u_{\eps}(\cdot; \bm{y})}_{1,\eps,D} \norm{\partial_{\bm{y}}^{\bm{\nu}} u_{\eps}(\cdot; \bm{y})}_{1,\eps,D} .
	\end{align*}
	This shows that
	\begin{equation}\label{eq:recursion-norm-partial-deriv}
		\norm{\partial_{\bm{y}}^{\bm{\nu}} u_{\eps}(\cdot; \bm{y})}_{1,\eps,D} \leq \sum_{j \in \supp(\bm{\nu})} \nu_j b_j \norm{\partial_{\bm{y}}^{\bm{\nu} - \bm{e}_j} u_{\eps}(\cdot; \bm{y})}_{1,\eps,D},
	\end{equation}
	where $b_j \coloneqq \frac{\norm{\psi_j}_{L^{\infty}(D)}}{\min\{1, b_{0}\}}$, for $j \geq 1$. Denote also $\bm{b} \coloneqq (b_j)_{j\geq 1}$. Finally, we show \eqref{eq:bound-partial-deriv} by induction on $\abs{\bm{\nu}}$. For $\abs{\bm{\nu}} = 0$ this corresponds to $\norm{u_{\eps}(\cdot; \bm{y})}_{1,\eps,D} \leq B \, \norm{f}_{-1,\eps,D}$, with the constant $B$ as in \eqref{eq:bound-sol}. For $\abs{\bm{\nu}} >0$, using \eqref{eq:recursion-norm-partial-deriv} and the induction hypothesis, we find
	\begin{align*}
		\norm{\partial_{\bm{y}}^{\bm{\nu}} u_{\eps}(\cdot; \bm{y})}_{1,\eps,D} &\leq B \, \norm{f}_{-1,\eps,D} \sum_{j \in \supp(\bm{\nu})} \nu_j b_j (\abs{\bm{\nu}} - 1)! \bm{b}^{\bm{\nu}-\bm{e}_j}\\
		 &= B \, \left(\sum_{j \in \supp(\bm{\nu})} \nu_j\right) (\abs{\bm{\nu}} - 1)! \, \bm{b}^{\bm{\nu}} \, \norm{f}_{-1,\eps,D} = B \, \abs{\bm{\nu}}! \, \bm{b}^{\bm{\nu}} \norm{f}_{-1,\eps,D}.
	\end{align*}
	
	We now turn to the case $\eps = 0$. Here, \eqref{eq:de} reduces to
	\begin{equation*}
		b(x; \bm{y}) u_0(x; \bm{y}) = f(x), 
	\end{equation*}
	where $f \in L^2(D)$. Since $b(\cdot; \bm{y}) \ge b_0$ a.e., this equation admits the unique solution
	\begin{equation*}
		u_0(x;\bm{y}) = \frac{f(x)}{b(x; \bm{y})} \in L^2(D).
	\end{equation*}
	Note that this is well-defined since $b(x; \bm{y}) > 0$ for a.e. $x \in D$ and every $\bm{y} \in U$.
	For any $\bm{\nu} \in \cF$, by the Leibniz rule for $u_0(x; \bm{y})$, we obtain
	\begin{equation*}
		\partial_{\bm{y}}^{\bm{\nu}} u_0(x; \bm{y}) =  (-1)^{\abs{\bm{\nu}}} \abs{\bm{\nu}}! f(x) \frac{\prod_{j \in \supp(\bm{\nu})} \psi_j(x)^{\nu_{j}}}{b(x;\bm{y})^{\abs{\bm{\nu}} + 1}}.
	\end{equation*}
	Using $b(x;\bm y)\geq b_0$ and the definition of $b_j$ gives the same bound as
	in \eqref{eq:bound-partial-deriv}.
\end{proof}
\section{Dimension truncation}
\label{sec:dim-trunc}
Our model for the parametric coefficient $b$ involves countably many parameters. Although the solution possesses the required parametric regularity, practical computations require truncating the expansion \eqref{ass:affine-param} to a finite number of terms. In this section, we therefore recall a general result concerning the truncation of smooth functions depending on countably many parameters to some finite dimension $s$.

Dimension truncation is frequently studied in the context of QMC. A setting closely related to ours was considered in \cite{kuo2012QMC}, where a corresponding dimension truncation estimate was derived. For affine-parametric operator equations, an improved convergence rate was later established in \cite{gantner2018dimension} by means of a Neumann series argument. This could likewise be applied to our model with an affine parametric reaction coefficient. However, this approach ultimately relies on the affine structure of the problem. Anticipating more general situations with potentially nonlinear parameter dependence, we instead use a recent result based solely on parametric regularity of the problem and a Taylor expansion argument; see \cite{guth2022application, guth2024lognormal}.

Let $U = [-1/2,1/2]^{\bN}$ and denote by
$\bm{\pi}(\dif{\bm{y}}) \coloneqq \bigotimes_{j\geq 1} \dif y_j$
the uniform product probability measure on $U$.
Let $X$ be a separable Banach space and let $g \in C^{\infty}(U; X)$. For $s \in \bN$, define the dimension-truncated map of dimension $s$ by
\[
	g_s(\bm{y}) \coloneqq g(\bm{y}_{\leq s}, \bm{0})
	\coloneqq g(y_1,\dots,y_s,0,0,\dots).
\]

We will work under the following assumptions:
\begin{enumerate}[label=\textbf{D.\arabic*}]
	\item\label{ass:dim-trunc-1} For $\bm{\pi}$-a.e. $\bm{y} \in U$
	\begin{equation*}
		\norm{g(\bm{y}) - g_s(\bm{y})}_{X} \longrightarrow 0,\quad \text{as } s \to \infty.
	\end{equation*} 
	\item\label{ass:dim-trunc-2} 
	There exist a non-increasing sequence $\bm{b} = (b_j)_{j \geq 1} \in \ell^p(\bN)$, with $b_j \ge 0$ and $0 < p <1$, and constants $\eta, \mu \ge 0$ such that,
	for every nonzero $\bm{\nu} \in \cF$,
	\begin{equation*}
		\norm{\partial_{\bm{y}}^{\bm{\nu}} g(\bm{y})}_{X} \leq \abs{\bm{\nu}}! \mu \eta^{\abs{\bm{\nu}}} \bm{b}^{\bm{\nu}},\qquad \forall \bm{y} \in U.
	\end{equation*}
\end{enumerate}
\begin{proposition}[{\cite[Theorem 4.3]{guth2024lognormal}}]
	\label{prop:dim-truncation-guth}
	Let $g(\bm{y}) \in X$ for all $\bm{y} \in U$. 
        Suppose that \cref{ass:dim-trunc-1,ass:dim-trunc-2} hold. 
        Then
	\begin{equation*}
		\norm*{\int_{U} (g(\bm{y}) - g_{s}(\bm{y})) \bm{\pi}(\dif \bm{y})}_{X} \leq C s^{- \frac{2}{p} + 1},
	\end{equation*} 
	where the constant $C > 0$ is independent of the dimension $s$.
	
	Moreover, for every $G \in X'$,
	\begin{equation*}
		\abs*{\int_{U} G(g(\bm{y}) - g_{s}(\bm{y})) \bm{\pi}(\dif \bm{y})} \leq C \norm{G}_{X'} s^{-\frac{2}{p} + 1},
	\end{equation*}
	where the constant $C >0$ is as above.
\end{proposition}
\begin{remark}
	The result in \cite{guth2024lognormal} shows that $C$ depends on
	$g$ only through the derivative-bound constants in \cref{ass:dim-trunc-2}
	and on the sequence $\bm b$. This observation is what yields
	$\eps$-robust dimension-truncation bounds below.
\end{remark}

We now apply \cref{prop:dim-truncation-guth} to
$g(\bm y)=u_{\eps}(\cdot;\bm y)$ with $X=V_{\eps}$ and norm
$\norm{\cdot}_{1,\eps,D}$.
\begin{corollary}\label{cor:dim-truncation}
	Let $f\in L^2(D)$, and suppose that \crefrange{ass:affine-param}{ass:ordering}
	hold with $0<p<1$ in \cref{ass:p-summability}. For any $0 < \eps \leq 1$ and $s \in \bN$, define
	\[
		u_{\eps,s}(\cdot;\bm y)
		\coloneqq u_{\eps}(\cdot;(\bm y_{\leq s},\bm{0})).
	\]
	Then,
	\begin{equation}\label{eq:dim-trunc-eq-1}
		\norm*{\int_U (u_{\eps}(\bm{y})-u_{\eps,s}(\bm{y})) \, \bm{\pi}(\dif{\bm{y}}) }_{1,\eps,D} \leq  C s^{-\frac{2}{p} + 1},
	\end{equation}
	where the constant $C > 0$ is independent of $s$ and $\eps$.
	
	Moreover, for every $G \in (V_{\eps})'$,
	\begin{equation}\label{eq:dim-trunc-eq-2}
		\abs*{\int_U G(u_{\eps}(\bm{y})-u_{\eps, s}(\bm{y})) \, \bm{\pi}(\dif{\bm{y}}) } \leq  C \norm{G}_{(V_{\eps})'} s^{-\frac{2}{p} + 1},
	\end{equation}
	with the same constant as in \eqref{eq:dim-trunc-eq-1}.
\end{corollary}
\begin{proof}
	For each fixed $\eps\in(0,1]$, \cref{thm:param-reg} gives the required
	smoothness of the map $\bm y\mapsto u_{\eps}(\cdot;\bm y)$ as a $V_{\eps}$-valued
	map. We first verify \cref{ass:dim-trunc-1}. Since
	\[
	\norm{b(\cdot;\bm y)-b(\cdot;\bm y_{\leq s},\bm{0})}_{L^{\infty}(D)}
	\leq \frac12\sum_{j>s}\norm{\psi_j}_{L^{\infty}(D)}\to0,
	\]
	the stability estimate~\eqref{eq:stab} implies
	\[
	\norm{u_{\eps}(\cdot;\bm y)-u_{\eps,s}(\cdot;\bm y)}_{1,\eps,D}\to0
	\qquad\text{for every }\bm y\in U.
	\]
	
	Next, \cref{thm:param-reg-bounds} gives, for every nonzero
	$\bm\nu\in\cF$,
	\[
	\sup_{\bm y\in U}
	\norm{\partial_{\bm y}^{\bm\nu}u_{\eps}(\cdot;\bm y)}_{1,\eps,D}
	\leq B\abs{\bm\nu}!\bm b^{\bm\nu}\norm{f}_{-1,\eps,D},
	\]
	with $B$ and $\bm b$ defined in~\eqref{eq:b-sequence-partial-deriv}. Since
	$f\in L^2(D)$, $\norm{f}_{-1,\eps,D}\leq\norm{f}_{L^2(D)}$, uniformly in
	$\eps$. \Cref{ass:p-summability} and \cref{ass:ordering} imply that
	this sequence $\bm b$ is non-increasing and belongs to $\ell^p(\bN)$. Hence \cref{ass:dim-trunc-2} holds with constants independent of $\eps$.
	
	\Cref{prop:dim-truncation-guth} therefore proves
	\eqref{eq:dim-trunc-eq-1}. Applying the same proposition to the bounded linear
	functional $G\in(V_{\eps})'$ gives~\eqref{eq:dim-trunc-eq-2}.
\end{proof}
\section{Higher-order QMC}
\label{sec:qmc}
This section recalls the higher-order quasi--Monte Carlo (QMC) result used later for the
dimension-truncated parametric problem. We consider the numerical approximation 
of integrals of functions $F$ defined on the $s$-dimensional unit cube $[0,1]^s$
\begin{equation*}
	I_s(F) \coloneqq \int_{[0,1]^s} F(\bm{y}) \dif\bm{y},
\end{equation*}
by QMC quadrature rules 
\begin{equation*}
	Q_{N,s}(F) \coloneqq \frac{1}{N} \sum_{n=0}^{N-1} F(\bm{y}_n),
\end{equation*}
where the quadrature points $\bm{y}_0,\dots,\bm{y}_{N-1} \in [0,1]^s$ are chosen according to a deterministic lattice rule belonging to the family of digital nets. In
the application to~\eqref{eq:integral-functional-sol}, these points are shifted
from $[0,1]^s$ to $[-1/2,1/2]^s$.
Integrals of this form arise, for instance, in the evaluation of linear functionals of the solution to the parametric problem \eqref{eq:de}--\eqref{eq:bc}.
\subsection{Weighted spaces and SPOD weights}
We use the higher-order unanchored Sobolev spaces from~\cite{dick2014higher} in
the case $r=\infty$, which is the only case needed below.\footnote{We follow the definition in \cite{dick2016multilevel}, which corrects a typographical error in the corresponding definition of \cite{dick2014higher}.} For $k\in\bN$, set
$\{1:k\}\coloneqq\{1,\ldots,k\}$. If $\mfu\subseteq\{1:s\}$ and
$\mfv\subseteq\mfu$, then
$(\bm{\alpha}_{\mfv},\bm{\tau}_{\mfu\setminus\mfv},\bm{0})$ denotes the multi-index
$\bm\nu\in\{0,\ldots,\alpha\}^s$ with $\nu_j=\alpha$ for $j\in\mfv$,
$\nu_j=\tau_j$ for $j\in\mfu\setminus\mfv$, and $\nu_j=0$ for $j\notin\mfu$.
\begin{definition}[Higher-order unanchored Sobolev norm]
	Let $\alpha, s \in \bN, 1 \leq q \leq \infty$, and $1 \leq r \leq \infty$, and let $\bgg = (\gamma_{\mfu})_{\mfu \subset \bN}$ be a collection of non-negative real weights.
	Assume further that $F \colon [0,1]^s \to \bR$ has partial derivatives of orders up to $\alpha$ with respect to each variable. Then we define the norm of $F$ by a higher-order unanchored Sobolev norm
	\begin{align*}
		\norm{F}_{s,\alpha,\bgg,q,r}
		&\coloneqq
		\left[
		\sum_{\mfu \subset \{1:s\}}
		\left(
		\frac{1}{\gamma_{\mfu}^{q}}
		\sum_{\mfv \subset \mfu}
		\sum_{\bm{\tau}_{\mfu \setminus \mfv}
			\in \{1:\alpha\}^{\abs{\mfu \setminus \mfv}}}
		\right.
		\right.\\
		&\qquad \left.
		\left.
		\int_{[0,1]^{\abs{\mfv}}}
		\abs*{
			\int_{[0,1]^{s - \abs{\mfv}}}
			\left(
			\partial_{\bm{y}}^{
				(
				\bm{\alpha}_{\mfv},
				\bm{\tau}_{\mfu \setminus \mfv},
				\bm{0}
				)
			} F
			\right)(\bm{y})
			\dif{\bm{y}_{\{1:s\}\setminus \mfv}}
		}^{q}
		\, \dif{\bm{y}_{\mfv}}
		\right)^{r/q}
		\right]^{1/r},
	\end{align*}
	with the obvious modifications if $q$ or $r$ is infinite. Let further $\mathcal{W}_{s,\alpha,\bgg,q,r}$ denote the Banach space of all such functions $F$ with finite norm. In the following we will exclusively work with the case $r = \infty$.
\end{definition}
We shall use this norm for integrands satisfying the derivative bound
\begin{equation}
	\label{eq:ass-QMC-integrand}
	\abs{\partial_{\bm y}^{\bm\nu}F(\bm y)}
	\leq C_1\abs{\bm\nu}!\bm b^{\bm\nu},
	\qquad
	\bm\nu\in\{0,1,\ldots,\alpha\}^s,
	\quad \bm y\in[0,1]^s,
\end{equation}
where $C_1>0$ and $\bm b=(b_j)_{j\geq1}$ is a non-increasing sequence in
$\ell^p(\bN)$ for some $0<p<1$.
\subsubsection*{SPOD weights}
In what follows, we identify one specific choice of weights $\bgg$ such that an integrand
$F$ satisfying bounds on the partial derivatives of the form \eqref{eq:ass-QMC-integrand} belongs to the corresponding
weighted spaces $\mathcal{W}_{s,\alpha,\bgg,q,\infty}$, with $1 \leq q \leq \infty$ arbitrary.
In addition, such a choice of weights should allow for an efficient construction of a lattice rule to approximate the integral $I_{s}(F)$.

One such example of weights, known as \emph{smoothness-driven product and order dependent weights} (SPOD), was first introduced in \cite{dick2014higher}.
For $1 \leq q < \infty$, the case $r=\infty$ and the derivative bound \eqref{eq:ass-QMC-integrand} give
\begin{align*}
	\norm{F}_{s,\alpha, \bgg,q,\infty}
	&\le
	C_1
	\max_{\mfu \subset \{1:s\}}
	\gamma_{\mfu}^{-1}
	\left[
	\sum_{\mfv \subset \mfu}
	\sum_{\bm{\tau}_{\mfu \setminus \mfv}
		\in \{1:\alpha\}^{\abs{\mfu \setminus \mfv}}}
	\left(
	\abs*{
		(
		\bm{\alpha}_{\mfv},
		\bm{\tau}_{\mfu \setminus \mfv},
		\bm{0}
		)
	}!
	\,
	\bm{b}^{
		(
		\bm{\alpha}_{\mfv},
		\bm{\tau}_{\mfu \setminus \mfv},
		\bm{0}
		)
	}
	\right)^q
	\right]^{1/q} \\
	&\leq
	C_1
	\max_{\mfu \subset \{1:s\}}
	\gamma_{\mfu}^{-1}
	\sum_{\mfv \subset \mfu}
	\sum_{\bm{\tau}_{\mfu \setminus \mfv}
		\in \{1:\alpha\}^{\abs{\mfu \setminus \mfv}}}
	\abs*{
		(
		\bm{\alpha}_{\mfv},
		\bm{\tau}_{\mfu \setminus \mfv},
		\bm{0}
		)
	}!
	\,
	\bm{b}^{
		(
		\bm{\alpha}_{\mfv},
		\bm{\tau}_{\mfu \setminus \mfv},
		\bm{0}
		)
	} \\
	&=
	C_1
	\max_{\mfu \subset \{1:s\}}
	\gamma_{\mfu}^{-1}
	\sum_{\bm{\nu}_{\mfu} \in \{1:\alpha\}^{\abs{\mfu}}}
	\abs{\bm{\nu}_{\mfu}}!
	\prod_{j \in \mfu}
	\left(2^{\delta(\nu_j,\alpha)} b_j^{\nu_j}\right).
\end{align*}
For $q=\infty$, the usual modification gives the same final bound.
Thus, to ensure $\norm{F}_{s,\alpha,\bgg,q,\infty} \leq C_1$, we choose the weights
\begin{equation}\label{eq:SPOD-weights}
	\gamma_{\mfu} \coloneqq
	\sum_{\bm{\nu}_{\mfu} \in \{1:\alpha\}^{\abs{\mfu}}}
	\abs{\bm{\nu}_{\mfu}}!
	\prod_{j \in \mfu}
	\left(2^{\delta(\nu_j, \alpha)} b_j^{\nu_j}\right),
\end{equation}
for each $\mfu \subset \{1:s\}$.

A direct consequence of the above is that an integrand $F$ satisfying the bounds \eqref{eq:ass-QMC-integrand} has bounded $\norm{\cdot}_{s,\alpha,\bgg,q,\infty}$-norm for arbitrary $\alpha, s \in \bN$ and $1 \leq q \leq \infty$, and
\begin{equation*}
	\norm{F}_{s,\alpha,\bgg,q,\infty} \leq C_1,
\end{equation*}
where $C_1$ is the constant from \eqref{eq:ass-QMC-integrand}.
\subsection{Interlaced polynomial lattice rules}
In this subsection, we recall results on higher-order QMC quadrature rules based on \emph{interlaced polynomial lattice rules}. These rules were introduced in~\cite{Dick2007Explicit, Dick2008Walsh} and belong to the class of higher-order digital nets; see also~\cite{dick2010digital} for a detailed discussion. Here we use the notation from \cite{dick2014higher}.

Let $\mathfrak{b}$ be prime, let $\bZ_{\mathfrak{b}}$ be the finite field with $\mathfrak{b}$ elements, and let $\bZ_{\mathfrak{b}}[x]$ denote the polynomials over $\bZ_{\mathfrak{b}}$.
\begin{definition}[Polynomial lattice point set]
	For a prime $\mathfrak{b}$ and any $m \in \bN$, let $P \in \bZ_{\mathfrak{b}}[x]$ be an irreducible polynomial with $\deg(P) = m$ (modulus). For a given dimension $s \geq 1$, select $s$ polynomials $q_1(x),\dots,q_s(x)$ from the set
	\begin{equation*}
		\cG_{\mathfrak{b},m} \coloneqq \{q(x) \in \bZ_{\mathfrak{b}}[x] \setminus \{0\} : \deg(q) < m\},
	\end{equation*} 
	and write
	\begin{equation*}
		\bm{q} = \bm{q}(x) = \left(q_1(x),\dots,q_s(x)\right) \in \cG_{\mathfrak{b},m} ^s.
	\end{equation*}
	Then, $\bm{q}$ is called the generating vector. For each integer $0 \leq n < \mathfrak{b}^m$, let $n = \eta_{0} + \eta_{1} \mathfrak{b} + \cdots + \eta_{m-1} \mathfrak{b}^{m-1}$ be the $\mathfrak{b}$-adic expansion of $n$, and we associate with $n$ the polynomial
	\begin{equation*}
		n(x) = \sum_{r=0}^{m-1} \eta_{r} x^{r} \in \bZ_{\mathfrak{b}}[x].
	\end{equation*}
	Furthermore, we denote by $v_m$ the map from $\bZ_{\mathfrak{b}}((x^{-1}))$ to the interval $[0,1)$ defined for any integer $w$ by
	\begin{equation*}
		v_m\left(\sum_{\ell = w}^{\infty} t_{\ell} x^{-\ell}\right) = \sum_{\ell = \max(1,w)}^{m} t_{\ell} \mathfrak{b}^{-\ell}.
	\end{equation*} 
	Then, the QMC point set $\bm{\cS}_{P,\mathfrak{b},m,s}(\bm{q})$ of a polynomial lattice rule comprises the points
	\begin{equation*}
		\bm{y}_n = \left(v_{m}\left(\frac{n(x) q_1(x)}{P(x)}\right),\dots,v_{m}\left(\frac{n(x) q_{s}(x)}{P(x)}\right)\right) \in [0,1)^s, \quad n=0,\dots,\mathfrak{b}^{m} - 1.
	\end{equation*}
\end{definition}
\begin{definition}[Digit interlacing]
	For an interlacing factor $\alpha\in\bN$, define
	$\cD_{\alpha}\colon[0,1)^{\alpha}\to[0,1)$ as follows. If
	$x_r=\sum_{a=1}^{\infty}\xi_{r,a} \mathfrak{b}^{-a}$ for $r=1,\ldots,\alpha$, then
	\[
	\cD_{\alpha}(x_1,\ldots,x_{\alpha})
	\coloneqq
	\sum_{a=1}^{\infty}\sum_{r=1}^{\alpha}
	\xi_{r,a} \mathfrak{b}^{-r-(a-1)\alpha}.
	\]
	For vectors in $[0,1)^{\alpha s}$, apply this map componentwise in consecutive
	blocks of length $\alpha$. An interlaced polynomial lattice rule of order
	$\alpha$ in $s$ dimensions uses the points
	$\cD_{\alpha}(\bm{\cS}_{P,\mathfrak{b},m,\alpha s}(\bm q))$, with
	$\bm q\in\cG_{\mathfrak{b},m}^{\alpha s}$.
\end{definition}
For functions in $\mathcal W_{s,\alpha,\bgg,q,\infty}$, the following CBC
construction and error bound are known.
\begin{theorem}[{\cite[Theorem~3.10]{dick2014higher}}]
	\label{thm:ho-qmc-error}
	Let $\alpha, s \in \bN$ with $\alpha > 1$, $1 \leq q \leq \infty$, and let $\bgg = (\gamma_{\mfu})_{\mfu \subset \bN}$ be a collection of weights. Let $\mathfrak{b}$ be prime and let $m \in \bN$ be arbitrary. Then, an interlaced polynomial lattice rule of order $\alpha$ with $N = \mathfrak{b}^m$ points $\{\bm{y}_0,\dots,\bm{y}_{N-1}\} \subset [0,1]^s$ can be constructed using a CBC algorithm, such that for all $\lambda \in \left(\frac{1}{\alpha}, 1\right]$
	\begin{equation}\label{eq:cbc-bound-estimate}
		\abs*{\frac{1}{\mathfrak{b}^m} \sum_{n=0}^{\mathfrak{b}^m - 1} F(\bm{y}_n) - \int_{[0,1]^s} F(\bm{y}) \dif{\bm{y}}} \leq \left(\frac{2}{\mathfrak{b}^m - 1} \sum_{\emptyset \neq \mfu \subset \{1:s\}} \gamma_{\mfu}^{\lambda} \left[\rho_{\alpha, \mathfrak{b}}(\lambda)\right]^{\abs{\mfu}} \right)^{1/\lambda} \norm{F}_{s,\alpha,\bgg,q,\infty},
	\end{equation}
	where \(\rho_{\alpha,\mathfrak b}(\lambda)\) is the constant from \cite[Theorem~3.10]{dick2014higher}.
\end{theorem}
\begin{remark}[CBC construction]
	\label{rem:cost-cbc}
	The efficient component-by-component (CBC) construction of lattice rules satisfying the bound in \eqref{eq:cbc-bound-estimate} has been studied extensively in the literature. Here we only mention \cite[Section 3.4]{dick2014higher}.
	
	Using the upper bound on the worst case error of the lattice rule as a search criterion allows for the construction of an interlaced polynomial lattice rule of order $\alpha$ with $N = \mathfrak{b}^m$ points with total computational cost
	\[
		\cO\bigl(\alpha sN\log N+\alpha^2s^2N\bigr).
	\]
\end{remark}
Combining \cref{thm:ho-qmc-error} with the SPOD weights yields the
dimension-independent rate used in the sequel. We refer to the discussion after
\cite[Theorem~3.10]{dick2014higher} on SPOD weights for details.
\begin{corollary}
	\label{cor:qmc-error-bound}
	Let $0<p<1$, let $\bm b\in\ell^p(\bN)$ be non-increasing, and set
	$\alpha=\lfloor1/p\rfloor+1$. Let $\bgg$ be the SPOD weights
	\eqref{eq:SPOD-weights}. For every $s\in\bN$, $1\leq q\leq\infty$, prime $\mathfrak{b}$,
	and $N=\mathfrak{b}^m$, an
	interlaced polynomial lattice rule of order $\alpha$ can be constructed such
	that
	\begin{equation}
		\label{eq:cbc-bound-estimate-spod}
		\abs{Q_{N,s}(F)-I_s(F)}
		\leq C(p,\mathfrak{b},\bm b)\norm{F}_{s,\alpha,\bgg,q,\infty}N^{-1/p},
	\end{equation}
	where $C(p,\mathfrak{b},\bm b)>0$ is independent of $s$ and $N$. In particular, if
	$F$ satisfies~\eqref{eq:ass-QMC-integrand}, then
	\[
		\abs{Q_{N,s}(F)-I_s(F)}
		\leq C(p,\mathfrak{b},\bm b)C_1N^{-1/p}.
	\]
\end{corollary}
\begin{remark}
	The constant in~\eqref{eq:cbc-bound-estimate-spod} is independent of the
	dimension $s$. The dependence on the particular integrand enters only through
	the bounds on its partial derivatives. 
\end{remark}
\section{Spatial approximation by Galerkin methods}
\label{sec:spatial-approx}
This section isolates the spatial discretization input needed in the final error
bound. The abstract assumption below is deliberately stated as an approximation
property: the particular finite element space and mesh may depend on the spatial
problem, but the rate must be uniform in the perturbation parameter and in the
parametric coefficient.
\begin{assumption}[Galerkin spatial approximation]
	\label{ass:spatial-approx}
	For each discretization parameter $h>0$, let $V_h \subset V_{\eps}$ be a finite-dimensional subspace. For $\eps \in (0,1]$ and 
	$\bm{y} \in U$, let $u^h_\eps(\cdot;\bm{y}) \in V_h$ be the Galerkin approximation defined by
	\[
	B_\eps(u_{\eps}^{h},v_h; \bm{y}) = \inp{f}{v_h}_{V_{\eps}' \times V_{\eps}},
	\qquad \forall v_h \in V_h.
	\]
	There exists a constant $C_{\text{FE}}>0$, independent of $\eps$, $\bm{y}$, and $h$,
	such that for every $\eps \in (0,1]$ and every $\bm{y} \in U$ the estimate
	\begin{equation}\label{eq:spatial-approx}
		\norm{u_{\eps}(\cdot;\bm{y}) - u^h_\eps(\cdot;\bm{y})}_{L^2(D)} \leq  C_{\text{FE}} \, \kappa(h)
	\end{equation}
	holds, where $\kappa(h)\to 0$ as $h \to 0$ describes the convergence rate
	of the Galerkin approximation in the $\norm{\cdot}_{L^2(D)}$--norm.
\end{assumption}
\begin{remark}[Dependence on domain and coefficients]
	The validity of~\eqref{eq:spatial-approx} depends on structural properties of the
	underlying problem, such as the shape of the domain, the regularity of the right-hand side,
	and the smoothness of the reaction coefficient $b$.
	Standard Galerkin methods on quasi-uniform meshes generally fail to provide $\eps$-uniform error estimates,
	as boundary and corner layers in $u_{\eps}$ are not resolved adequately. Robust estimates are available, for example, for layer-adapted
	$h$-FE methods on Shishkin or Bakhvalov meshes and for $hp$-FE methods under
	analyticity assumptions; see, e.g.,~\cite{RST2nd,SunStynes95,melenk1998hp,melenk2002hp}.
\end{remark}
\subsection{A one-dimensional model problem}
\label{subsec:1d-model-problem}
We now verify \cref{ass:spatial-approx} for a concrete one-dimensional model problem.
Let $D = I \coloneqq [-1,1]$. We assume further that $b(\cdot;\bm{y}) \in H^2(I)$ for all $\bm{y} \in U$, as well as
$f \in H^2(I)$.

The weak formulation is: find $u_{\eps}\in H_0^1(I)$ such that
\begin{equation}
	\label{eq:BVPweak-1d}
	B_{\eps}(u_{\eps},v)=F(v),
	\qquad \forall v\in H_0^1(I),
\end{equation}
where
\[
	B_{\eps}(w,v)
	\coloneqq \int_I \left(\eps^2 w'(x)v'(x)+b(x)w(x)v(x)\right)\dif x,
	\qquad
	F(v)\coloneqq \int_I f(x)v(x)\dif x.
\]
We assume
\begin{equation}
	\label{eq:bbounds-1d}
	0<b_0\leq b(x)\leq b_1<\infty
	\qquad \text{for a.e. }x\in I,
\end{equation}
with constants independent of $\bm y$. This is the one-dimensional version of \cref{ass:bbounds}. 
The corresponding energy norm satisfies
\[
	\norm{v}_{E,I}^2\coloneqq B_{\eps}(v,v)
	\simeq \norm{v}_{1,\eps,I}^2,
\]
with equivalence constants depending only on $b_0$ and $b_1$.

The regularity and asymptotic behavior of solutions to singularly perturbed reaction--diffusion problems of this type
have been studied extensively; see, for example,~\cite{RST2nd} and references therein.

In the present setting, we rely on a robust finite element approximation
result from~\cite{ManuskriptDLBL} for singularly perturbed
reaction--diffusion problems of the form \eqref{eq:BVPweak-1d}, which was formulated under slightly weaker assumptions on the coefficient $b$ and the data $f$ than those typically imposed in the literature.
The key insight for our purposes is that the constants in these estimates
depend on the reaction coefficient $b(x;\bm y)$ only through the global
bounds $b_0$ and $b_1$, and in particular are independent of the
parameter $\bm y \in U$. This allows us to derive FE error estimates that are \emph{uniform} in $\bm y$.
Accordingly, in what follows we suppress the explicit $\bm y$-dependence
of $b$ and consider the model problem with a coefficient
$b \colon I \to \bR$.
\subsection{Robust finite element approximation}
The Galerkin approximations of \eqref{eq:de}--\eqref{eq:bc}
will be based on the weak formulation \eqref{eq:BVPweak-1d}.
To define the discretization of \eqref{eq:BVPweak-1d},
for each $M \in \bN$,
let $\Delta = \{ x_i \}_{i=0}^{M}$ be a
partition of $I$ with $-1 = x_0 < x_1 < \cdots < x_M = 1$
and set
$I_i = (x_{i-1} , x_{i} )$, $h_i = x_{i} - x_{i-1}, i=1, \ldots, M, h_{\max} = \max_{1\leq i \leq M} h_i$.
With $\bP_1(I)$ the space of polynomials of degree at most $1$ on $I$,
we define
\begin{align} \label{eq:Sp}
	\mathcal{S}^1(\Delta) &\coloneqq \{ w \in H^1(I) : w|_{I_j} \in \bP_1(I_j), j=1,\dots,M \},\\
	\mathcal{S}^1_0(\Delta) &\coloneqq \mathcal{S}^1(\Delta) \cap H^1_0(I) \label{eq:Sp0}.
\end{align}
The Galerkin discretization of \eqref{eq:BVPweak-1d} is then given by:
find $u^\Delta_{\eps} \in \mathcal{S}^1_0(\Delta)$
such that
for all $v\in \mathcal{S}^1_0(\Delta)$,
\begin{equation}\label{eq:discrete}
	B_\eps(u^\Delta_{\eps},v)
	=
	\int_{-1}^1 \left\{ \eps^2 \left(u^\Delta_{\eps}\right)' v' + b u^\Delta_{\eps} v \right\} \dif x = \int_{-1}^1 f v \dif x.
\end{equation}
Galerkin orthogonality and coercivity give the best-approximation estimate
\[
	\norm{u_{\eps}-u_{\eps}^{\Delta}}_{E,I}
	\leq C\inf_{v\in\mathcal S_0^1(\Delta)}
	\norm{u_{\eps}-v}_{E,I},
\]
where $C$ depends only on $b_0$ and $b_1$.
\subsection{Shishkin meshes}
Standard finite element methods on uniform meshes do not yield approximations that are robust with respect to the parameter $\eps$.
Solutions of singularly perturbed problems tend to vary rapidly near the boundary, which discrete approximations on uniform meshes cannot resolve adequately.
To overcome this difficulty and obtain $\eps$-uniform approximations, one may employ \emph{layer-adapted meshes}.
Here, we only consider one of the simplest examples of such meshes, the Shishkin mesh~\cite{Shishkin1989}. There exist many other layer-adapted meshes used in the context of singularly perturbed problems such as 
eXp meshes \cite{XenoThesis} and Bakhvalov meshes \cite{Bakhvalov1969}, to name but a few. 
To obtain an $\eps$-uniform approximation, we employ a standard Shishkin
mesh $\Delta_S$ on $I=[-1,1]$ with transition point
\[
	\lambda=\min\{\theta\eps\ln M_{h},1/4\},
	\qquad \theta = \frac{4}{\sqrt{b_0}}.
\]
Assuming that $M_{h}$ is divisible by $4$, the mesh consists of $M_{h}/4$,
$M_{h}/2$, and $M_{h}/4$ uniform sub-intervals on the left layer region, the
coarse interior region, and the right layer region, respectively.
The following result is proved in \cite[Proposition~3.1]{ManuskriptDLBL}.
\begin{proposition}[Approximation on Shishkin mesh]
	\label{prop:shish}
	For $0<\eps \leq 1$ and $b,f \in H^1(I)$,
	let $u_{\eps}$ be the solution to \eqref{eq:de}--\eqref{eq:bc}.
	For $M_{h} \in \mathbb{N}$ divisible by $4$, set $h \coloneqq M_{h}^{-1}$ and let $u_{\eps}^{h}\in \mathcal{S}^1_0(\Delta_S)$ be the
	continuous piecewise linear (CpwL) FE approximation of \eqref{eq:discrete} on the Shishkin mesh $\Delta_S$ with $M_{h}$ intervals
	at length scale $\eps \in (0,1]$.
	There exists a positive constant $C$, independent of $\eps$ and $M_{h}$,
	such that
	for every $0<\eps \leq 1$ and
	for every $4 \leq M_{h} \leq 1/\eps$, it holds
	\begin{equation}\label{eq:ShiskBd}
		\Vert u_{\eps} - u_{\eps}^{h} \Vert_{1,\eps,I} \leq C M_{h}^{-1} \ln M_{h}.
	\end{equation}
	If in addition $b, f \in H^2(I)$, then the bound \eqref{eq:ShiskBd} also holds for $M_{h} > 1/\eps$.
\end{proposition}
\begin{remark}
	\label{rem:kappa-shishkin}
	\Cref{prop:shish} verifies \cref{ass:spatial-approx} for the
	one-dimensional model problem on Shishkin meshes with
	\[
	\kappa(h)=h\ln(1/h), \qquad h=M_h^{-1}.
	\]
	Indeed, $\norm{v}_{L^2(I)}\leq\norm{v}_{1,\eps,I}$, and the constant in
	\eqref{eq:ShiskBd} is independent of $\eps$ and $h$.
\end{remark}
\subsubsection{Improved error bounds in the \texorpdfstring{$L^2$}{L2}-norm}
The $L^2$ interpolation error bound obtained in the proof of \cref{prop:shish} suggests that one might expect a higher order of convergence in the $L^2$-norm of the FE approximation than the rate shown in \cref{prop:shish}.
A standard Aubin--Nitsche argument cannot be applied here, however, because the interpolation estimates rely on a boundary layer decomposition of $u_{\eps}$, and an analogous decomposition is not readily available for the associated dual problem.
Nevertheless, an improved $L^2$-convergence rate was proved in~\cite[Theorem~5.3]{SunStynes95}.
For completeness, we adapt that argument to our setting and provide a proof below.
\begin{proposition}
	\label{prop:l2-improved-rate}
	For $0 < \eps \leq 1$, and $f,b \in H^2(I)$, let $u_{\eps}$ be the solution to \eqref{eq:de}--\eqref{eq:bc}. For $M_{h} \in \mathbb{N}$ divisible by $4$, set $h \coloneqq M_{h}^{-1}$ and let
	$u_{\eps}^{h} \in \mathcal{S}_{0}^1(\Delta_{S})$ be the CpwL FE approximation of \eqref{eq:discrete} on the Shishkin mesh at length scale $\eps \in (0,1]$ with $M_{h}$ intervals. Then for every $0 < \eps \leq 1$ and for every $M_{h} \geq 4$ it holds
	\begin{equation}\label{eq:improved-l2-rate-FE-interpol}
		\norm{\mathcal{I}_{M_{h}} u_{\eps} - u_{\eps}^{h}}_{L^2(I)}  \leq \frac{b_{1}}{b_{0}} \norm{\mathcal{I}_{M_{h}} u_{\eps} - u_{\eps}}_{L^2(I)}.
	\end{equation}
	and
	\begin{equation}\label{eq:improved-l2-rate-FE-sol}
		\norm{u_{\eps} - u_{\eps}^{h}}_{L^2(I)}  \leq C \left( 1 + \frac{b_{1}}{b_{0}} \right) h^2 \ln^2(1/h),
	\end{equation}
	where $C >0$ is independent of $\eps$ and $M_{h}$. Here $\mathcal{I}_{M_{h}}$ denotes the nodal interpolation operator on the Shishkin mesh.
\end{proposition}
\begin{proof}
	Set
	\[
		\eta \coloneqq \mathcal{I}_{M_{h}} u_{\eps} - u_{\eps},
		\qquad
		\xi \coloneqq \mathcal{I}_{M_{h}} u_{\eps} - u_{\eps}^{h} \in \mathcal{S}_0^1(\Delta_S).
	\]
	Since $b \ge b_0$ a.e. in $I$, Galerkin orthogonality gives
	\begin{align}
		\begin{split}\label{eq:aubin-nitsche-coerc}
			b_0 \norm{\xi}_{L^2(I)}^2
			\leq (b \xi, \xi)_{L^2(I)} \leq B_{\eps}(\xi,\xi) = B_{\eps}(\eta, \xi) + B_{\eps}(u_{\eps} - u_{\eps}^{h}, \xi)
			= B_{\eps}(\eta, \xi).
		\end{split}
	\end{align}
	On each mesh interval, $\xi$ is affine and $\eta$ vanishes at the
	endpoints. Hence elementwise integration by parts yields
	$(\eps^2\eta',\xi')_{L^2(I)}=0$, and therefore
	\[
		B_\eps(\eta,\xi)
		= (b\eta,\xi)_{L^2(I)}
		\le b_1\norm{\eta}_{L^2(I)}\norm{\xi}_{L^2(I)}.
	\]
	Dividing by $b_0\norm{\xi}_{L^2(I)}$ if $\xi\neq0$ proves
	\eqref{eq:improved-l2-rate-FE-interpol}; the case $\xi=0$ is trivial.
	This shows \eqref{eq:improved-l2-rate-FE-interpol}. Finally, we note that
	\begin{equation*}
		\norm{u_{\eps} - u_{\eps}^{h}}_{L^2(I)} \leq \norm{u_{\eps} - \mathcal{I}_{M_{h}} u_{\eps} }_{L^2(I)}  + \norm{\mathcal{I}_{M_{h}} u_{\eps} - u_{\eps}^{h}}_{L^2(I)}  \leq \left(1 + \frac{b_{1}}{b_{0}}\right) \norm{\mathcal{I}_{M_{h}} u_{\eps} - u_{\eps}}_{L^2(I)}.
	\end{equation*}
	The estimate \eqref{eq:improved-l2-rate-FE-sol} then follows from the $L^2$ interpolation error estimate in the proof of \cref{prop:shish}.
\end{proof}
\begin{remark}
	\label{rem:kappa-shishkin-l2}
	In the notation of \cref{ass:spatial-approx}, \cref{prop:l2-improved-rate}
	yields the sharper one-dimensional Shishkin-mesh rate
	\[
	\kappa(h)=h^2\ln^2(1/h), \qquad h=M_h^{-1}.
	\]
	This improved $L^2$ rate is the spatial approximation rate used in the
	one-dimensional fully discrete QMC--FE estimate below.
\end{remark}
\section{Combined error analysis}
\label{sec:combined-error}
We now combine the dimension truncation, QMC quadrature, and spatial
discretization estimates. Throughout this section, the QMC rule is applied on
$[0,1]^s$ and then shifted to the parameter domain $[-1/2,1/2]^s$.
\begin{theorem}
	Let $f\in L^2(D)$, and suppose that
	\crefrange{ass:affine-param}{ass:ordering} hold with $0<p<1$ in
	\cref{ass:p-summability}. Assume also that the spatial discretization
	satisfies \cref{ass:spatial-approx} with discretization parameter $h>0$.
	Let $\alpha=\lfloor1/p\rfloor+1$ and let $N=\mathfrak{b}^m$ with $\mathfrak{b}$ prime.
	
	Then, for any $0<\eps\leq1$, every $s\in\bN$, every $h>0$, and every
	$G\in L^2(D)$, if we approximate the integral
	\eqref{eq:integral-functional-sol} by the QMC method
	\eqref{eq:qmc-rule-sol} using an interlaced polynomial lattice rule of
	order $\alpha$ with $N$ points in $s$ dimensions, and use the Galerkin
	approximation $u_{\eps,s}^{h}\in V_h$ for the truncated spatial problem,
	then
	\begin{equation}\label{eq:combined-bound-general}
		\abs{I(G(u_{\eps})) - Q_{N,s}(G(u_{\eps,s}^{h}))}
		\leq C\left(N^{-1/p}+s^{-2/p+1}+\kappa(h)\right)
		\norm{G}_{L^2(D)},
	\end{equation}
	where $C>0$ is independent of $\eps$, $s$, $N$, and $h$, but may depend on
	fixed data norms and coefficient bounds.
\end{theorem}
\begin{proof}
	Let
	\[
	F_s(\bm z)\coloneqq
	G\left(u_{\eps,s}\left(\cdot;\bm z-\frac{\bm{1}}{\bm{2}}\right)\right),
	\qquad \bm z\in[0,1]^s,
	\]
	and define $F_s^h$ analogously with $u_{\eps,s}^h$.
	Following \cite[Section 4.2]{dick2016holomorphic}, we decompose the error as
	\begin{align*}
		\abs{I(G(u_{\eps}))-Q_{N,s}(F_s^h)}
		&\leq
		\abs{I(G(u_{\eps}))-I_s(F_s)}
		+\abs{I_s(F_s)-Q_{N,s}(F_s)} \\
		&\quad +\abs{Q_{N,s}(F_s-F_s^h)}
		\eqqcolon E_I+E_{II}+E_{III}.
	\end{align*}
	For $E_I$, \cref{cor:dim-truncation} gives
	\[
	E_I\leq C\norm{G}_{(V_{\eps})'}s^{-2/p+1}
	\leq C\norm{G}_{L^2(D)}s^{-2/p+1},
	\]
	where the embedding $L^2(D)\hookrightarrow(V_{\eps})'$ has norm at most one.
	For $E_{II}$, the chain rule and \cref{thm:param-reg-bounds} imply, for
	all $\bm\nu\in\{0,1,\ldots,\alpha\}^s$,
	\begin{align*}
		\abs{\partial_{\bm z}^{\bm\nu}F_s(\bm z)}
		&\leq \norm{G}_{(V_{\eps})'}
		\norm{\partial_{\bm y}^{\bm\nu}u_{\eps,s}
			(\cdot;\bm z-\bm{1}/\bm{2})}_{1,\eps,D} \\
		&\leq C\abs{\bm\nu}!\bm b^{\bm\nu}
		\norm{f}_{L^2(D)}\norm{G}_{L^2(D)}.
	\end{align*}
	Thus $F_s$ satisfies~\eqref{eq:ass-QMC-integrand}, and
	\cref{cor:qmc-error-bound} yields
	\[
		E_{II}\leq C N^{-1/p}\norm{G}_{L^2(D)}.
	\]
	Here and below, constants may depend on the fixed data norms, but not on
	$\eps$, $s$, $N$, or $h$.
	For the final term $E_{III}$, we observe that since $Q_{N,s}$ is an equal weight rule, we can bound the FE approximation error at each point $\bm{y}_n$ by 
	the worst-case error and obtain
	\begin{align*}
		E_{III}
		&\leq \frac1N\sum_{n=0}^{N-1}
		\abs*{G\left(u_{\eps,s}(\cdot;\bm y_n-\bm{1}/\bm{2})
			-u_{\eps,s}^h(\cdot;\bm y_n-\bm{1}/\bm{2})\right)} \\
		&\leq \norm{G}_{L^2(D)}
		\sup_{\bm y\in[-1/2,1/2]^s}
		\norm{u_{\eps,s}(\cdot;\bm y)-u_{\eps,s}^h(\cdot;\bm y)}_{L^2(D)} \\
		&\leq C\norm{G}_{L^2(D)}\kappa(h).
	\end{align*}
	Combining the three estimates proves~\eqref{eq:combined-bound-general}.
\end{proof}
For the one-dimensional model problem, the improved $L^2$ spatial estimate from \cref{prop:l2-improved-rate} gives the following concrete rate.
\begin{corollary}[One-dimensional model problem]
	Let $f\in H^2(I)$ and suppose that \crefrange{ass:affine-param}{ass:ordering}
	hold. In addition, assume that $b(\cdot;\bm y)\in H^2(I)$ for every
	$\bm y\in U$. Let $M_h$ denote the number of intervals in the Shishkin mesh and
	set $h=M_h^{-1}$.
	
	Then, for any $G\in L^2(I)$ and every $\eps\in(0,1]$, if we approximate
	\eqref{eq:integral-functional-sol} by \eqref{eq:qmc-rule-sol} using an
	interlaced polynomial lattice rule of order $\alpha=\lfloor1/p\rfloor+1$ with
	$N=\mathfrak{b}^m$ points in $s\geq1$ dimensions, and use a CpwL FE approximation on the Shishkin mesh with $M_h$ intervals, then
	\begin{equation}\label{eq:combined-bound-1d}
		\abs{I(G(u_{\eps})) - Q_{N,s}(G(u_{\eps,s}^{h}))}
		\leq C\left(N^{-1/p}+s^{-2/p+1}+h^2\ln^2(1/h)\right)
		\norm{f}_{L^2(I)}\norm{G}_{L^2(I)}.
	\end{equation}
	Equivalently, the spatial term is bounded by $C M_h^{-2}\ln^2 M_h$. The constant $C>0$ is
	independent of $\eps$, $s$, $N$, and $h$.
\end{corollary}
\begin{proof}
	The bound follows directly from \eqref{eq:combined-bound-general} by using \cref{rem:kappa-shishkin-l2}.
\end{proof}
The total cost of evaluating $Q_{N,s}(G(u_{\eps,s}^{h}))$ is $\cO(s N M_h)$ operations. Therefore, if we choose $s, N$ and $M_h$ to balance the terms in \eqref{eq:combined-bound-1d} such that the combined error is of order $\cO(\delta)$ for some $\delta > 0$, the resulting QMC--FE approximation has an 
overall cost of order
\begin{equation*}
	\cO(sNM_h)
	=\cO\left(\ln(1/\delta)\,
	\delta^{-\left(p+\frac{p}{2-p}+\frac12\right)}\right).
\end{equation*}
\begin{remark}
	In typical applications, the cost of evaluating $Q_{N,s}(G(u_{\eps,s}^{h}))$ outweighs the cost of the construction of the lattice rule. In the context of the above one-dimensional model problem, this is the case as long as
	\[
		M_h \geq \max \{\alpha^2 s, \alpha \ln(N)\}.
	\]
\end{remark}
\section{Numerical experiments}
\label{sec:numerics}
We present a numerical example that fits into the affine-parametric framework,
and study the convergence of the FE approximation, dimension
truncation and QMC--FE approximation.

All experiments are performed for the following values of the perturbation parameter $\eps$:
\[
	\eps\in\{10^{-3},10^{-5},10^{-7},10^{-9}\}.
\]
In all cases we observe
essentially the same convergence rates. We emphasize that only a single set of QMC
points needs to be generated for all values of $\eps$.

We return to the one-dimensional model problem introduced in
\cref{subsec:1d-model-problem}. For
$\bm y\in U=[-1/2,1/2]^{\bN}$ and $I=[-1,1]$, consider
\eqref{eq:de}--\eqref{eq:bc} with uncertain reaction coefficient
\begin{equation}
	\label{eq:reaction-coeff-experiment}
	b(x;\bm y)=2+\theta\sum_{j=1}^{\infty}y_{j} j^{-\zeta}\sin(\pi jx).
\end{equation}
Here $\theta$ controls the size of the fluctuation coefficients; in the plots
below we use $\theta=0.1$. For this choice, the coefficient is uniformly
positive since
\[
	\abs*{\theta\sum_{j=1}^{\infty}y_{j} j^{-\zeta}\sin(\pi jx)}
	\leq \frac{\theta}{2}\sum_{j=1}^{\infty}j^{-\zeta}<1.
\]
For every finite $s\geq1$, the $s$-term expansion is smooth in the spatial
variable $x$, but not uniformly so in $s$. For the sine basis, a uniform
$H^2(I)$ bound follows from the square-summability condition
$\sum_{j\geq1}j^{4-2\zeta}<\infty$, i.e. from $\zeta>5/2$. This uses the
structured Fourier basis and should be distinguished from absolute summability
in $H^2(I)$. In the following we choose $\zeta=21/8>5/2$.

Writing
$\beta_j \coloneqq \theta \norm{j^{-\zeta}\sin(\pi j x)}_{L^\infty(I)}$,
we have $\beta_j \simeq b_j$ for $\bm{b}$ as in
\eqref{eq:b-sequence-partial-deriv}, up to a multiplicative constant
depending only on the uniform lower bound in
\eqref{eq:reaction-coeff-experiment}.
It holds that
$(\beta_j)_{j\geq1}\in\ell^p(\bN)$ for every $p>1/\zeta$. Thus, for
$\zeta=21/8$, we may take $p=8/21+\delta$ with arbitrarily small
$\delta>0$, and the QMC construction uses
$\alpha=\lfloor1/p\rfloor+1=3$.

As the forcing term we take
\begin{equation*}
	f(x)\coloneqq\frac{(\max(x,0))^2}{2}\in H^2(I),
\end{equation*}
and as quantity of interest in~\eqref{eq:combined-bound-1d}, we consider the
integral of the parametric solution over the physical domain,
\[
	G(u_{\eps}(\bm y))=\int_{I} u_{\eps}(x;\bm y)\dif x.
\]

The QMC points are chosen according to interlaced polynomial lattice rules with
$N=2^m$ points, generated using the \texttt{QMC4PDE} Python
package~\cite{Nuyens2024QMC4PDE}. The construction of the generating matrices is
based on a fast CBC algorithm for SPOD weights from~\cite{kuo2016survey}.
Setting the parameters $a_1=0$, $a_2=1$, $a_3=0$, $\alpha=3$, and $d_1=1$ in
the \texttt{QMC4PDE} package corresponds to the weights in
\eqref{eq:SPOD-weights} and the derivative bounds in
\eqref{eq:ass-QMC-integrand}.

The spatial discretization of~\eqref{eq:de}--\eqref{eq:bc} is performed using
$\bP_1$ finite elements on Shishkin meshes.
\subsection{FE convergence}
The finite element approximation on Shishkin meshes for one-dimensional singularly
perturbed reaction--diffusion problems is classical. Since the forcing term $f$
in our model problem has lower regularity than is commonly assumed in the
literature, we include a brief convergence study.

To verify the spatial discretization error predicted by the theory, we compute
the FE approximation on Shishkin meshes with increasing numbers of intervals.
We denote by $u_{\eps}^{h_\ell}$ the FE solution obtained on a Shishkin mesh
with $M_\ell=2^\ell$ intervals, for $\ell=2,\dots,19$. As reference solution we
use the computation with $M_{\mathrm{ref}}=2^{20}$ intervals.

As reaction coefficient we choose $b(\cdot;\bm y_0)$ in
\eqref{eq:reaction-coeff-experiment}, where
$\bm y_0=(-1/2,-1/2,\ldots)$.

We report three types of errors: the $1,\eps$-error
\[
	\norm{u_{\eps}^{h_{\mathrm{ref}}}(\cdot;\bm y_0)
		-u_{\eps}^{h_\ell}(\cdot;\bm y_0)}_{1,\eps,I},
	\qquad \ell=2,\dots,19,
\]
as well as the $L^2$-error and the balanced-norm error. For
$\eps\in(0,1]$ and $v\in H^1(I)$, the balanced norm is defined as
\[
	\norm{v}_{\eps,b,I}^2
	\coloneqq \eps\norm{v'}_{L^2(I)}^2+\norm{v}_{L^2(I)}^2.
\]
Theoretical results for FE approximations on Shishkin meshes in the balanced
norm are available in the literature; see, e.g.,~\cite{LinssCAX2024,MX2016}.
The results are shown in \cref{fig:FEM-conv}.
\begin{figure}[t]
	\centering
	\includegraphics[width=\textwidth]{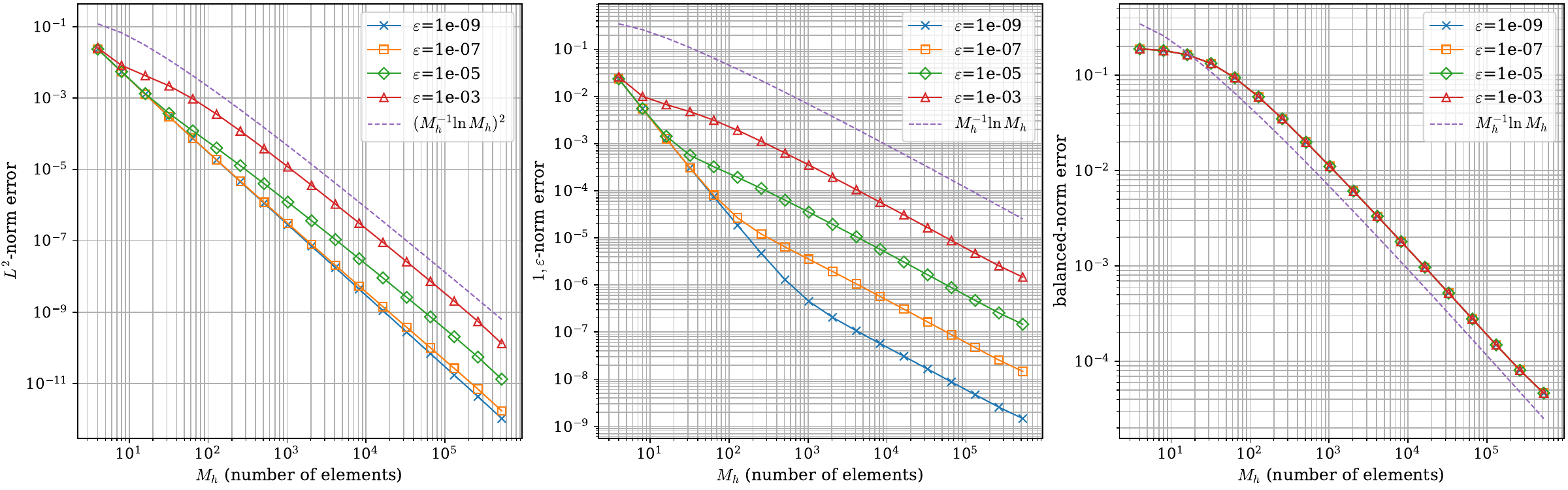}
	\caption{
		Finite element convergence on Shishkin meshes for four values of $\eps$. Left: $L^2$-norm error; middle: $1,\eps$-norm error; right: balanced-norm error. The dashed lines indicate the expected reference convergence rates.}
	\label{fig:FEM-conv}
\end{figure}
\subsection{Dimension truncation convergence}
\label{sec:DimTrncCnv}
In order to numerically verify the theoretical dimension truncation
bound~\eqref{eq:dim-trunc-eq-2}, we compute the QMC--FE approximation for
increasing truncation dimension $s$, while fixing the remaining discretization
parameters. The number of QMC points and the number of finite element intervals
are chosen as $N=4096$ and $M_{h}=2^{20}$, respectively, so that the
quadrature and spatial discretization errors are negligible. A reference value
is computed using $s_{\mathrm{ref}}=256$ dimensions.

We vary the truncation dimension according to $s_\ell=2^\ell$,
$\ell=0,\dots,7$, and compute the approximate integral dimension truncation
error
\[
	\abs{Q_{N,s_{\mathrm{ref}}}(G(u_{\eps,s_{\mathrm{ref}}}^{h}))
		-Q_{N,s_\ell}(G(u_{\eps,s_\ell}^{h}))},
	\qquad \ell=0,\dots,7.
\]
The results are shown in \cref{fig:dim-trunc-conv}.

From the theory in \cref{sec:dim-trunc}, 
we expect an $\eps$-robust convergence rate of
\[
	\cO(s^{-2/p+1})=\cO(s^{-17/4+\delta'}),
\]
for arbitrarily small $\delta'>0$. As shown in
\cref{fig:dim-trunc-conv}, the observed errors decay at this rate after
the first two data points.
\begin{figure}[t]
	\centering
	\begin{subfigure}[t]{0.48\textwidth}
		\centering
		\includegraphics[width=\linewidth]{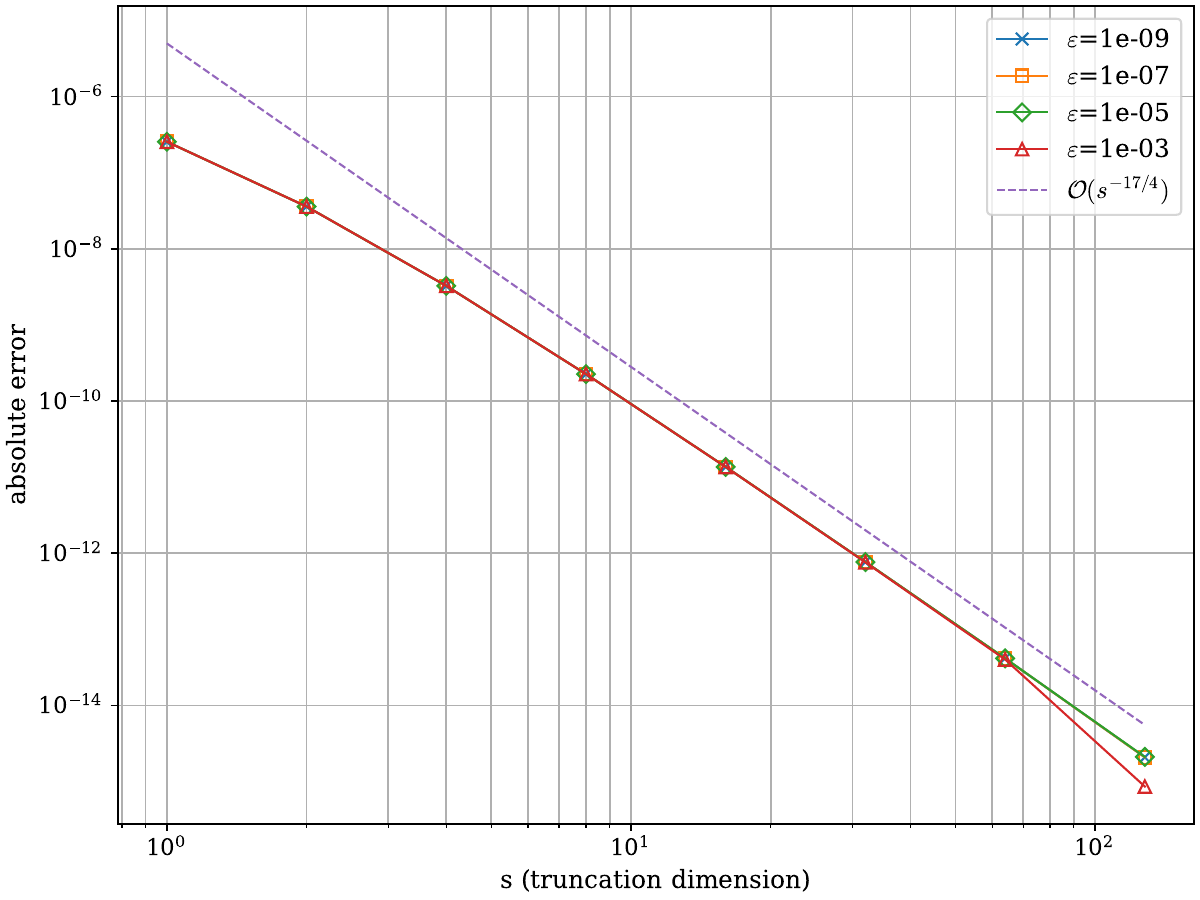}
		\caption{Dimension truncation convergence.}
		\label{fig:dim-trunc-conv}
	\end{subfigure}
	\hfill
	\begin{subfigure}[t]{0.48\textwidth}
		\centering
		\includegraphics[width=\linewidth]{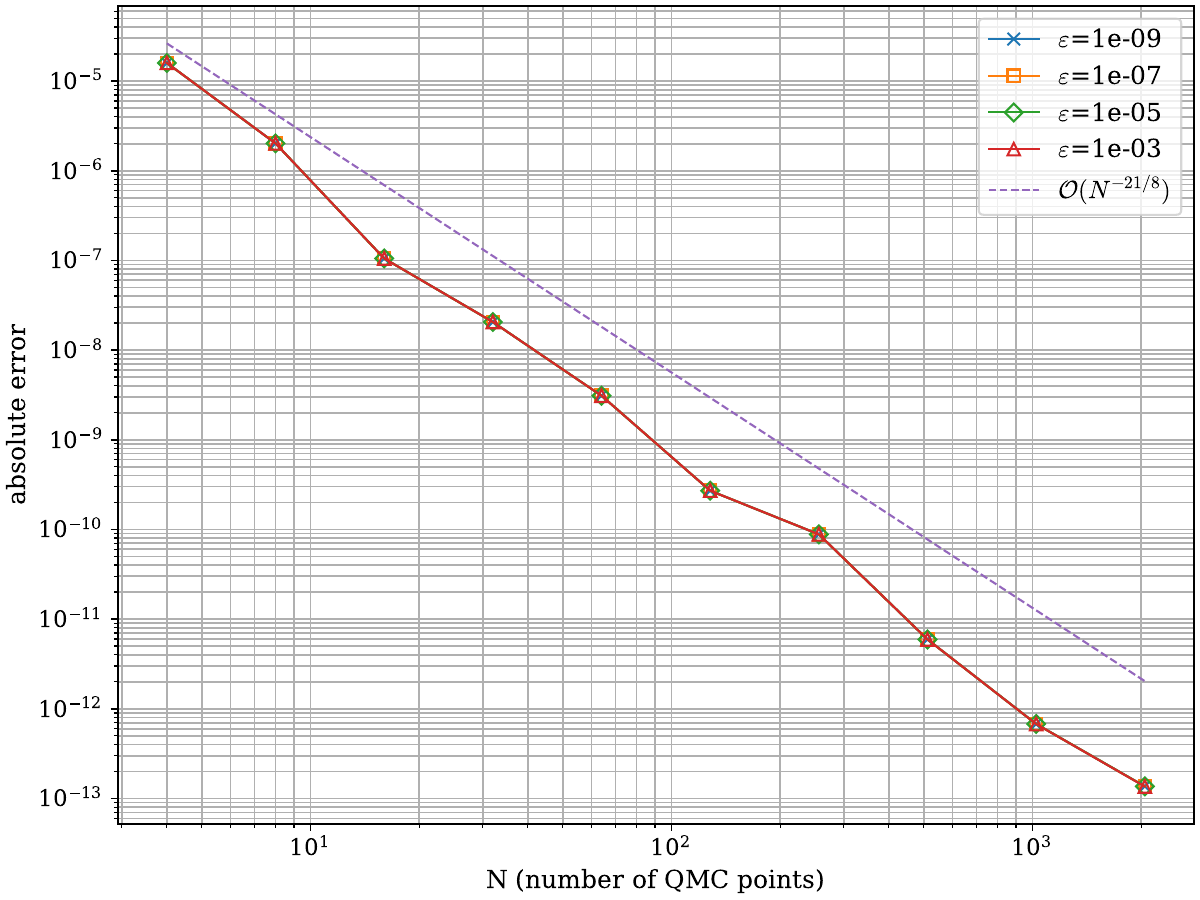}
		\caption{QMC--FE convergence.}
		\label{fig:qmc-conv}
	\end{subfigure}
	\caption{Dimension truncation and QMC--FE convergence for four values of $\eps$ 
		for the reaction coefficient in \eqref{eq:reaction-coeff-experiment},
		with fluctuation amplitudes decaying as $j^{-\zeta}$, where $\zeta=21/8$.
		Left: absolute error due to truncation of the parametric expansion at
		dimension $s$; the dashed line indicates the expected decay rate
		$\mathcal{O}(s^{-17/4})$.
		Right: absolute QMC--FE error obtained with $N$ QMC points; the
		dashed line indicates the expected decay rate
		$\mathcal{O}(N^{-21/8})$.}
	\label{fig:conv-results}
\end{figure}
\subsection{QMC--FE convergence}
\label{sec:QMC-FEc}
Similarly, to study the combined QMC--FE error~\eqref{eq:combined-bound-1d},
we compute the QMC--FE approximation for increasing numbers of QMC points $N$,
while fixing the truncation dimension and spatial discretization. We take
$s=256$ and $M_{h}=2^{20}$ so that the corresponding errors are negligible, and use
the same reference value with $N_{\mathrm{ref}}=2^{12}$ as in the dimension truncation experiment.

We vary the number of QMC points according to $N_\ell=2^\ell$,
$\ell=1,\dots,11$, and compute the approximate QMC--FE integration error
\[
\abs*{Q_{N_{\mathrm{ref}},s}(G(u_{\eps,s}^{h}))
	-Q_{N_\ell,s}(G(u_{\eps,s}^{h}))},
\qquad \ell=1,\dots,11.
\]
The results are shown in \cref{fig:qmc-conv}.

We observe a convergence rate that is higher than predicted by
\eqref{eq:cbc-bound-estimate-spod}. One possible explanation is that the proof
of \cref{cor:qmc-error-bound} requires the choice $\lambda \ge p$ in
order to ensure that the sum on the right-hand side of
\eqref{eq:cbc-bound-estimate} is bounded independently of the dimension $s$.
However, for fixed truncation dimension $s$, one could instead take
$\lambda=1/\alpha+\delta$ for arbitrarily small $\delta>0$, which yields an
analogous bound to~\eqref{eq:cbc-bound-estimate} with rate
$N^{-\alpha+\delta}$, at the expense of a constant that depends on $s$.
\section{Conclusion}
\label{sec:concl}
%
We developed a QMC--FE error analysis of a linear elliptic singularly perturbed
reaction--diffusion equation with an uncertain reaction coefficient. 
Assuming affine-parametric dependence of this coefficient on a parameter sequence
$\bm{y} = (y_j)_{j\geq 1} \in U = [-1/2,1/2]^{\mathbb{N}}$,
we analyzed the convergence rates of QMC integration for the approximation of the
statistical mean with respect to the uniform product probability
measure on $U$. 
We established the parametric regularity of solutions 
with explicit dependence on both the parameter sequence $(y_j)_{j\geq 1}$ 
and the singular perturbation parameter $\eps$.

Under the abstract assumption of the availability of a parameter-robust Galerkin finite element discretization, 
we obtained algebraic convergence rates for the fully discrete approximation method that are uniform with respect to the parameter sequence $(y_j)_{j\geq 1}$ 
characterizing the uncertain reaction coefficient, and robust with respect to the singular perturbation parameter $\eps \in (0,1]$.

The present error analysis was developed for the 
linear scalar model problem \eqref{eq:de}--\eqref{eq:bc}
in order to develop the mathematical framework and to verify
in a direct manner, the $\eps$-independence of the constants in all
parts of the analysis.

In light of previous work on the numerical analysis of UQ for PDEs, 
several extensions of the present analysis suggest themselves.
First, we considered only single-level QMC--FE methods. 
A natural extension of the present work
is the analysis of multilevel QMC--FE discretizations.  
As is the case for nonsingularly perturbed PDEs with uncertain inputs, 
such discretizations offer the prospect of 
improved error versus numerical work for statistical quantities of the parametric
solutions; see, for example, \cite{Kuo2015Multilevel}. 
For singularly perturbed problems such as \eqref{eq:de}--\eqref{eq:bc}, 
however, 
the $\eps$-explicit regularity analysis would need to be refined accordingly in order to
obtain $\eps$-robust MLQMC--FE error bounds.

The present error analysis considered the $\eps$-explicit PDE regularity 
and FE approximation error in $D$ only in an $\eps$-weighted energy norm. 
As is well known (e.g. \cite{LinssCAX2024,MX2016}), for certain applications
the FE error analysis in so-called $\eps$-dependent, \emph{balanced norms} 
is of interest. Error bounds that are uniform with respect to $\bm{y}$ and robust
with respect to $\eps$, analogous to the presently obtained 
results for the energy norm error, also hold in such norms.
Details will be reported elsewhere.

Further generalizations include non-affine parametric 
reaction coefficients, in particular coefficients with Gevrey-regular parametric dependence.
These extensions will also be considered elsewhere.

{\small 
	\printbibliography
}
\end{document}